\documentclass[oneside,reqno,11pt,a4paper]{amsart}
\usepackage[T1]{fontenc}
\usepackage[margin=2.4cm]{geometry}
\usepackage{amsmath,amsfonts,amsbsy,amssymb,amscd,amsthm}
\usepackage[usenames,dvipsnames,svgnames]{xcolor}
\usepackage[utf8]{inputenc}
\usepackage[ruled]{algorithm2e}
\usepackage{graphicx}
\usepackage{svg}
\usepackage{caption}
\usepackage{subcaption}
\usepackage{doi}
\usepackage[foot]{amsaddr}
\usepackage{makecell}
\usepackage{hyperref}
\usepackage{acronym}
\usepackage{listings}
\usepackage{textgreek}
\usepackage{url}
\usepackage{color}
\usepackage{framed}
\usepackage{verbatim}
\usepackage{fancyvrb}
\usepackage{stmaryrd}
\usepackage{newtxtext}
\usepackage{newtxmath}
\usepackage{microtype} 
\usepackage[final]{pdfpages}
\usepackage{tikz}
\usepackage{booktabs}
\usepackage{multirow}
\usepackage[backend=biber,
sorting=none,
defernumbers=false,
maxbibnames=5,
style=numeric-comp,
isbn=false,
bibencoding=utf8,
safeinputenc, 
url=false,
doi=true,
giveninits=true]{biblatex}
\hypersetup{
breaklinks=true,
bookmarksopen=true,
pdftitle={Adaptive Finite Element Interpolated Neural Networks},    
pdfauthor={Santiago Badia, Wei Li and Alberto F. Mart{\'{\i}}},     
colorlinks=true,       
linkcolor=blue,          
citecolor=blue,        
filecolor=black,      
urlcolor=blue           
}  
\definecolor{bg}{rgb}{0.93,0.93,0.93}
  
\newtheorem{theorem}{Theorem}[section]

\newtheorem{proposition}[theorem]{Proposition}
\newtheorem{corollary}[theorem]{Corollary}
\newtheorem{definition}[theorem]{Definition}

\newtheorem{remark}[theorem]{Remark}
\newtheorem{example}[theorem]{Example}

\acrodef{cnn}[CNN]{convolutional neural network}
\acrodef{pde}[PDE]{partial differential equation}
\acrodef{rhs}[RHS]{right-hand side}
\acrodef{fe}[FE]{finite element}
\acrodef{fem}[FEM]{finite element method}
\acrodef{dof}[DoF]{degree of freedom}
\acrodef{nn}[NN]{neural network}
\acrodefplural{dof}[DoFs]{degrees of freedom}
\acrodef{pinn}[PINN]{physics-informed \ac{nn}}
\acrodef{vpinn}[VPINN]{variational \ac{pinn}}
\acrodef{ivpinn}[IVPINN]{interpolated \ac{vpinn}}
\acrodef{feinn}[FEINN]{\ac{fe} interpolated \ac{nn}}
\acrodef{gmg}[GMG]{geometric multigrid}
\acrodef{spd}[SPD]{symmetric positive definite}
\acrodef{amr}[AMR]{adaptive mesh refinement}
\acrodef{ane}[ANE]{adaptive network enhancement}
\acrodef{dfr}[DFR]{deep Fourier residual}

\graphicspath{{../plots}{./}}
\newcommand{\tnor}[1]{{\left\vert\kern-0.25ex\left\vert\kern-0.25ex\left\vert #1 
\right\vert\kern-0.25ex\right\vert\kern-0.25ex\right\vert}}

\newcommand{\fig}[1]{Fig.~\ref{#1}}
\newcommand{\tab}[1]{Tab.~\ref{#1}}
\newcommand{\sect}[1]{Sec.~\ref{#1}}

\newcommand{\norm}[1]{\left\lVert #1 \right\rVert}
\newcommand{\ltwonorm}[1]{\left\lVert #1 \right\rVert _{L^2(\Omega)}}
\newcommand{\honenorm}[1]{\left\lVert #1 \right\rVert _{H^1(\Omega)}}

\newcommand{\argmin}[1]{\underset{#1}{\mathrm{arg\,min}}\,}

\begin{document}

\title[Adaptive finite element interpolated neural networks]{Adaptive finite element interpolated neural networks}
\author{Santiago Badia$^{1,2,*}$}
\email{santiago.badia@monash.edu}
\author{Wei Li$^1$}
\email{wei.li@monash.edu}
\author{Alberto F. Mart\'{\i}n$^3$}
\email{alberto.f.martin@anu.edu.au}
\address{$^1$ School of Mathematics, Monash University, Clayton, Victoria 3800, Australia.}
\address{$^2$ Centre Internacional de M\`etodes Num\`erics a l'Enginyeria, Campus Nord, UPC, 08034, Barcelona, Spain.}
\address{$^3$ School of Computing, The Australian National University, Canberra ACT 2600, Australia.}
\address{$^*$ Corresponding author.}


\begin{abstract}
The use of neural networks to approximate partial differential equations (PDEs) has gained significant attention in recent years. However, the approximation of PDEs with localised phenomena, e.g., sharp gradients and singularities, remains a challenge, due to ill-defined cost functions in terms of pointwise residual sampling or poor numerical integration. In this work, we introduce $h$-adaptive finite element interpolated neural networks. The method relies on the interpolation of a neural network onto a finite element space that is gradually adapted to the solution during the training process to equidistribute a posteriori error indicator. The use of adaptive interpolation is essential in preserving the non-linear approximation capabilities of the neural networks to effectively tackle problems with localised features. The training relies on a gradient-based optimisation of a loss function based on the (dual) norm of the finite element residual of the interpolated neural network. 
Automatic mesh adaptation (i.e., refinement and coarsening) is performed based on a posteriori error indicators till a certain level of accuracy is reached. The proposed methodology can be applied to indefinite and nonsymmetric problems. We carry out a detailed numerical analysis of the scheme and prove several a priori error estimates, depending on the expressiveness of the neural network compared to the interpolation mesh.  Our numerical experiments confirm the effectiveness of the method in capturing sharp gradients and singularities for forward PDE problems, both in 2D and 3D scenarios. We also show that the proposed preconditioning strategy (i.e., using a dual residual norm of the residual as a cost function) enhances training robustness and accelerates convergence.  
\end{abstract}

\keywords{neural networks, PINNs, finite elements, PDE approximation, $h$-adaptivity}

\maketitle
 
\section{Introduction}
\Acp{pde} are an effective mathematical tool for modelling a variety of physical phenomena in science and engineering, such as heat conduction, fluid dynamics and electromagnetism.
Since analytical solutions of \acp{pde} are only available in few special cases, numerical methods, such as the \ac{fem}, become the primary approach for approximating \ac{pde} solutions using computers. The \ac{fem} has solid mathematical foundations~\cite{Ern2021} and it can readily handle a broad range of \ac{pde} problems. Over the years, advanced discretisation techniques have been proposed~\cite{Arnold2006}, and highly efficient (non)linear solvers that can exploit large-scale supercomputers have been designed and polished~\cite{Brune2015,Badia2016,Drzisga2017}.

The \ac{fem} is a linear method that provides the best approximation in some specific measure on a finite-dimensional vector space. However, this space is not adapted to local features (e.g., sharp gradients or singularities) and the convergence is slow for problems that exhibit multiple scales. Adaptive \acp{fe}, e.g., the $h$-adaptive \ac{fem}, adapt the space to the solution by gradually increasing resolution in regions with larger \emph{a posteriori} error estimations. It adds an external loop, known as the solve, estimate, mark and refine loop, to the standard \ac{fem} workflow. Robust \emph{a posteriori} error estimators have been designed for diverse problems (see e.g.,~\cite{Ainsworth1997,Schoeberl2008}), making the adaptive \ac{fem} a widely employed technique in scientific disciplines.

Recently, deep learning, especially \acp{nn}, has demonstrated outstanding proficiency in many challenging tasks, including biomedical image segmentation~\cite{Ronneberger2015}, natural language processing~\cite{Collobert2011} and speech recognition~\cite{Graves2013}, just to mention a few. The striking success achieved by deep learning has encouraged researchers to explore the application of \acp{nn} in solving \acp{pde}. 
In contrast to standard discretisation methods, \acp{nn} are parametrised manifolds that enjoy non-linear approximation power, i.e., \acp{nn} have built-in adaptation capabilities essential in multi-scale problems.
Notably, \acp{pinn} have been proposed in~\cite{Raissi2019} as a novel approach to tackle \ac{pde}-constrained problems. The \ac{pinn} method integrates the \ac{pde} into the loss function used to train the \ac{nn}. This loss evaluates the strong \ac{pde} residual on collocation points, and the derivatives within this loss can be computed using automatic differentiation.
\Acp{pinn} are easy to implement and have been examined in different types of \acp{pde}, such as fractional~\cite{Pang2019} and stochastic \acp{pde}~\cite{Zhang2019}.

Many recent works have tried to address the limitations of \acp{pinn}, aiming to improve their robustness, efficiency and accuracy. The authors in~\cite{Kharazmi2021} introduce \acp{vpinn}, which make use of a polynomial test space and take the norm of the corresponding residual vector as the loss function. To overcome the integration challenge in \acp{vpinn}, the authors in~\cite{Berrone2022} propose the \ac{ivpinn} method. This method interpolates a \ac{nn} onto a \ac{fe} space and uses this polynomial interpolation in the variational residual.
In \acp{pinn}, the imposition of Dirichlet boundary conditions is a challenging task; as opposed to \ac{fe} spaces, \acp{nn} are not interpolatory objects. The addition of a penalty term to weakly impose the Dirichlet boundary conditions has a very negative effect on the training. In~\cite{Berrone2022}, the authors compose the \ac{nn} with a polynomial function to design a \ac{nn} that vanishes on the boundary. This approach is applicable to simple geometries but its extension to general geometries is unclear \cite{Sukumar2022}. Besides, while better than the penalty method, such composition negatively impacts the generalisation of \acp{nn} themselves~\cite{Badia2024}.

Our previous work~\cite{Badia2024} introduces the \ac{feinn} method. This method interpolates the \ac{nn} onto a \ac{fe} space that strongly satisfies the boundary condition, avoiding the issues discussed above. Besides, we propose to use the dual norm of the \ac{pde} residual as the loss function, which can be re-stated as a preconditioned loss function that is well-posed at the continuous level; see also~\cite{Rojas2023} for the application of dual norms in \acp{vpinn}. The preconditioning requires the application of the inverse of the inner product matrix (e.g., a Laplacian in $H^1(\Omega)$-conforming cases), which we efficiently approximate using geometric multigrid techniques. The \ac{feinn} method has been successfully applied to, e.g., forward \acp{pde} with smooth or non-smooth solutions and inverse heat conduction design problems. We show that the trained \acp{nn} can outperform the \ac{fem} solution by several orders of magnitude when the exact solution is smooth.

Despite these efforts, the training of \acp{pinn} and related methods often fail when dealing with target solutions containing high-frequency modes~\cite{Wang2022}, e.g., solutions with sharp gradients or discontinuities.
Some works aim to improve the expressivity of the \ac{nn} to capture these solutions, adding more neurons/layers~\cite{Cai2022}, or adapting the non-linear activation functions~\cite{Jagtap2020}.  
However, none of these works solve the integration issue; a quadrature on a uniform mesh is used in~\cite{Cai2022} and random collocation points in~\cite{Jagtap2020}.
To solve this issue, several approaches propose an adaptive distribution of points in collocation-based \acp{pinn}.
The authors in~\cite{Lu2021} design a residual-based adaptive refinement process to add more collocation points in the locations with the largest \ac{pde} residual, while a deep generative model that guides the sampling of collocation points is proposed in~\cite{Tang2023}. The adaptive collocation point movement method in~\cite{Hou2023} resamples the collocation points and uses an adaptive loss weighting strategy. In~\cite{Mao2023}, in addition to the strong residual of the \ac{pde}, the gradient information of the \ac{nn} is also used to guide the sampling procedure. All these adaptive methods have demonstrated enhanced performance and robustness compared to standard \acp{pinn} when solving \acp{pde} with steep gradients. Still, they suffer the applicability and robustness issues of collocated \acp{pinn} mentioned above. The provably accurate adaptive quadratures in~\cite{Magueresse2024} address the integration issue of \acp{nn} by automatically creating an accurate quadrature rule for loss functions that involve the integration of \ac{pde} residuals. However, the computation of the quadrature is computationally expensive beyond 2D. 

A posteriori error estimations are crucial for adaptive methods. Since the \ac{nn} in our method is interpolated onto a \ac{fe} space, we can readily leverage the well-studied \ac{fe} error estimators~\cite{Ainsworth1997,Kelly1983}. In recent years, \ac{nn}-based a posteriori error estimators have been developed. In~\cite{Minakowski2023}, the authors propose a \ac{nn} error estimator based on the dual weighted residual method for \acp{pinn}. Another related work is~\cite{Berrone2022_err_analysis}, where an error estimator for \acp{vpinn} consisting of a residual-type term, a loss-function term, and data oscillation terms is introduced. 
Functional error control techniques have been derived in~\cite{Muzalevskiy2023} to estimate the error in deep Galerkin simulations.

Remarkably, the state-of-the-art contains few investigations into the application of \acp{nn} for problems with singular solutions. In~\cite{Berrone2022,Badia2024}, the authors use the \ac{fe} interpolation of the trained \acp{nn} as the final solution of a 2D \ac{pde} with singularity. The \ac{dfr} method~\cite{Taylor2024} tackles problems with singular solutions using \acp{nn} via adaptive domain decomposition. However, the method is still prototypical in the sense that {automatic} refinement is currently limited to 1D scenarios and relies on tensor products of 1D orthonormal basis functions and an overkill integration procedure is used. 

In our previous study~\cite{Badia2024}, we have demonstrated \acp{feinn} outstanding performance in solving various forward and inverse problems. In this work, we carry out a detailed numerical analysis of the method for the \ac{fe} residual minimisation in a dual norm. Since the training of \acp{feinn} on a fixed underlying mesh limits their ability to leverage the full potential of \acp{nn} non-linear capabilities, we propose the $h$-adaptive \ac{feinn} method to overcome this limitation by introducing an \emph{automatic} adaptation of the interpolation space (mesh) during training, driven by \emph{a posteriori} error estimation. We propose a \ac{nn} error estimator based on the strong \ac{pde} residual (which is straightforward to implement and can be accurately computed via Gaussian quadratures) and a standard \ac{fe} Kelly indicator. These advances allow for a more comprehensive exploitation of the non-linear nature of \acp{nn}, thus enhancing the capability of the method to tackle problems with localised features. The minimisation of the loss function relies on a train, estimate, mark, adapt procedure, which is repeated till a certain level of accuracy is reached. We leverage hierarchically adapted non-conforming octree-based meshes and adaptive \ac{fe} spaces with hanging node constraints to keep the conformity of the interpolation.

We conduct a comprehensive set of numerical experiments, focusing on forward problems featuring steep gradients in 2D and singularities in both 2D and 3D. 
We verify that expressive enough \acp{nn} can replicate the rate of convergence of the $h$-adaptive \ac{fem} solutions. In addressing the sharp-gradient problem, the generalisation capabilities of the trained \acp{nn} yield remarkable results. The non-interpolated \ac{nn} solutions can be orders of magnitude more accurate than their interpolation counterparts on the same mesh. This observation is consistent with the findings in our previous work~\cite{Badia2024} for smooth solutions. Our findings reveal that the performance of the \ac{nn} error indicator is on par with, if not surpassing, the traditional Kelly \ac{fe} error estimator in various situations. Moreover, we confirm that a well-posed residual norm in the loss function has a tremendously positive impact on the convergence of the $h$-adaptive training process.

The outline of the article is the following. \sect{sec:method} gives details on the method that we propose, the model problem that we tackle, the \ac{fe} discretisation, the \ac{nn} architecture, and the loss functions we aim to minimise in the \ac{feinn} discretisation. In \sect{sec:analysis} we perform relevant numerical analysis of the interpolated \acp{nn} on a fixed \ac{fe} space and discuss the generalisation error of the \acp{nn}. \sect{sec:implementation} describes the implementation of the $h$-adaptive \ac{feinn} method and \sect{sec:experiments} presents numerical experiments on several forward problems with sharp gradients and singularities. Finally, \sect{sec:conclusions} draws conclusions and lists potential directions for further research.
\section{Methodology} \label{sec:method}
\subsection{Continuous problem} \label{subsec:method-contprob}

Let us consider a Hilbert space $\tilde U$ on a Lipschitz polyhedral domain $\Omega \subset \mathbb{R}^d$. We can consider a partition of the boundary $\partial \Omega$ into a Dirichlet $\Gamma_D$ and Neumann $\Gamma_N$ component. Let $U$ be the subspace of $\tilde{U}$ with zero trace on $\Gamma_D$. We denote by $(\cdot,\cdot)_U: U \times U \rightarrow \mathbb{R}$ an inner product in $U$ and by $\|\cdot\|_U$ the norm induced by the inner product. We can define an abstract weak formulation of a \ac{pde} problem: find 
\begin{equation} \label{eq:abstract-weak_form}
  u \in U \ : \ a(u,v) = \ell(v) - a(\bar{u},v), \quad \forall v \in U,
\end{equation}
where $a: U \times U \rightarrow \mathbb{R}$ is a bilinear form and $\ell \in U'$ is a functional in the dual space of $U$. Non-homogeneous Dirichlet boundary conditions are handled by a lifting $\bar{u} \in \tilde{U}$ while Neumann boundary conditions and body forces are included in $\ell$. The final solution of the problem is $u + \bar{u} \in \tilde{U}$. In order for this problem to be well-posed, the following inf-sup condition must be satisfied: there exists a constant $\beta_0 > 0$ such that
\begin{equation}\label{eq:continuous-infsup}
  \inf_{w \in U} \sup_{v \in U} \frac{a(w,v)}{\|w\|_U\|v\|_U} \geq \beta_0.
\end{equation} 
Besides, the bilinear form must be continuous, i.e., there exists a constant $\gamma > 0$ such that
\begin{equation}\label{eq:continuity}
a(u,v) \leq \gamma \|u\|_U\|v\|_U, \quad \forall u,v \in U.
\end{equation}
The solution of (\ref{eq:abstract-weak_form}) vanishes the \ac{pde} residual
\begin{equation} \label{eq:conv_diff_react_weak_residual}
  \mathcal{R}({u}) \doteq \ell(\cdot) - a({u} + \bar{u},\cdot) \in {U}',
\end{equation}
In this work, we propose an $h$-adaptive \ac{feinn} formulation to discretise (\ref{eq:abstract-weak_form}), which can be applied to any \ac{pde} that fits in the previous framework. It covers not only coercive problems but also indefinite and nonsymmetric problems. We consider linear \acp{pde} in the analysis, but the method can readily be applied to non-linear problems. We stress, however, that we focus on \emph{low-dimensional} \acp{pde}, e.g., $d\in \{2, 3\}$. The numerical approximation of high-dimensional \acp{pde} face other problems, due to the curse of dimensionality, and the requirements are different.

\subsection{Finite element method} \label{subsec:method-fem}

We consider a shape-regular partition $\mathcal{T}_h$ of $\Omega$.
On $\mathcal{T}_h$, we can define a trial \ac{fe} space $\tilde{U}_h \subset \tilde{U}$ and the subspace $U_h \subset U$ with zero trace. We define a \ac{fe} interpolant $\pi_h : \mathcal{C}^0 \rightarrow U_h$ obtained by evaluation of the \acp{dof} of $U_h$. 
Similarly, we can define the interpolant $\tilde{\pi}_h$ onto $\tilde{U}_h$. 
We can readily define a \ac{fe} lifting $\bar{u}_{h} \in \tilde{U}_h$ by interpolation on $\Gamma_D$. 

Let us consider a Petrov-Galerkin discretisation of (\ref{eq:abstract-weak_form}) with a test space $V_h \subset U$. The \ac{fe} problem reads: find $\tilde{u}_h = \bar{u}_h + u_h$ where
\begin{equation} \label{eq:diffusion_fe_weak_form}
  u_h \in U_h \ : \ a(u_h, v_h) = \ell(v_h) - a(\bar{u}_h,v_h) \doteq \bar{\ell}(v_h), \quad \forall v_h \in V_h.
\end{equation}
Problem (\ref{eq:diffusion_fe_weak_form}) is well-posed under the assumption that the trial and test \ac{fe} spaces $U_h$, $V_h$ have the same dimension and satisfy the following inf-sup condition \cite{Griffiths1978}: there exists a constant $\beta > 0$ independent of the mesh size $h$ such that  
\begin{equation}\label{eq:discrete-infsup}
  \inf_{w_h \in U_h} \sup_{v_h \in V_h} \frac{a(w_h,v_h)}{\|w_h\|_U\|v_h\|_U} \geq \beta, \qquad \beta > 0.
\end{equation} 
We represent with $\mathcal{R}_h$ the restriction $\mathcal{R}|_{{U}_h \times V_h}$. 

\subsection{Neural networks} \label{subsec:method-nn}
In this work, we use fully-connected feed-forward \acp{nn}. These kind of networks are constructed through a series of affine transformations and non-linear activation functions. However, the only requirement in our analysis is that the output can be written as a linear combination of the neurons at the previous layer, allowing for the application of other architectures as well. We represent the \ac{nn} architecture by a tuple $(n_0, \ldots n_L)\in \mathbb{N}^{(L+1)}$, where $L$ denotes the total number of layers and $n_k$ means the number of neurons in layer $k$ with $0 \leq k \leq L$. We take $n_0 = d$ and, for the scalar-valued \acp{pde} we consider, we have $n_L = 1$.

We denote by $\pmb{\Theta}_k: \mathbb{R}^{n_{k-1}} \to \mathbb{R}^{n_k}$ the affine map from layer $k-1$ to $k$ ($1 \leq k \leq L$), defined as $\pmb{\Theta}_k \pmb{x} = \pmb{W}_k \pmb{x} + \pmb{b}_k$, where $\pmb{W}_k \in \mathbb{R}^{n_k \times n_{k-1}}$ represents the weight matrix and $\pmb{b}_k \in \mathbb{R}^{n_k}$ the bias vector.
The activation function $\rho: \mathbb{R} \to \mathbb{R}$ is applied element-wise after each affine map, excluding the final one. With these definitions in place, the \ac{nn} can be conceptualised as a parametrisable function $\mathcal{N}(\pmb{\theta}): \mathbb{R}^d \to \mathbb{R}$, expressed as: 
\begin{equation} \label{eq:nn_structure}
  \mathcal{N}(\pmb{\theta}) = \pmb{\Theta}_L \circ \rho \circ \pmb{\Theta}_{L-1} \circ \ldots \circ \rho \circ \pmb{\Theta}_1,
\end{equation}
where $\pmb{\theta}$ denotes the collection of all trainable parameters $\pmb{W}_k$ and $\pmb{b}_k$. Besides, we apply the same activation function across all layers. 
In this work, we represent the \ac{nn} architecture as $\mathcal{N}$ and a specific instance of the \ac{nn} with parameters $\pmb{\theta}$ as $\mathcal{N}(\pmb{\theta})$. We extend the notation to also use $\mathcal{N}$ to represent the manifold of all possible realisations of the \ac{nn}, i.e., $\left\{ \mathcal{N}(\pmb{\theta}) \ : \ \pmb{\theta} \in \mathbb{R}^{n_{\mathcal{N}}} \right \}$, $n_\mathcal{N}$ being the total number of parameters. 

The \ac{nn} structure in (\ref{eq:nn_structure}) can be split into two parts: the last-layer operation, $\pmb{\Theta}_L$, defined by parameters $\pmb{\theta}_{ll} \in \mathbb{R}^{n_{L}}$, and the hidden-layer operations, 
$\pmb{\Theta}_{L-1} \circ \ldots \circ \rho \circ \pmb{\Theta}_1$, defined by parameters $\pmb{\theta}_{hl} \in \mathbb{R}^{n_{hl}}$, with $n_{hl}=n_{\mathcal{N}}-n_{L}$.  We refer to the input of the last-layer operation as the {\em last-layer functions}, consisting of $n_{L-1}$ output functions from the hidden-layer operations.
The output of a \ac{nn} realisation can be expressed as a linear combination of the last-layer functions, with coefficients $\pmb{W}_L$, plus a constant, the bias $\pmb{b}_L$. Thus, the hidden-layer parameters $\pmb{\theta}_{hl}$ define a linear space, i.e.,  the one spanned by the last-layer functions. We denote this linear space as $S(\pmb{\theta}_{hl})$, and refer to it as to the last-layer space.
With this notation, the manifold of \ac{nn} realisations can alternatively be expressed as the union of all possible last-layer vector spaces:
\begin{equation}\label{eq:nn_manifold-span}
\mathcal{N} = \bigcup_{\pmb{\theta}_{hl} \in \mathbb{R}^{n_{hl}}} S(\pmb{\theta}_{hl}).
\end{equation}

In the \ac{feinn} method, we apply the $U_h$ interpolation operator $\pi_h(\cdot)$ to the \ac{nn} realisations. Since the interpolation operator is linear, we can define the last-layer interpolated space $S_h(\pmb{\theta}_{hl}) \subset U_h$ as the span of the interpolation of the last-layer functions. The interpolated \ac{nn} manifold $\mathcal{N}_h \doteq \left\{ \pi_h(v) \ : \ v \in \mathcal{N} \right\} \subset U_h$ can be represented as the union of all last-layer interpolated spaces. These definitions are used in \sect{sec:analysis}.  

\subsection{Finite element interpolated neural networks} \label{subsec:method-feinn}
The \ac{feinn} method combines the \ac{nn} in~\eqref{eq:nn_structure} and the \ac{fe} problem in~\eqref{eq:diffusion_fe_weak_form}. The goal is to find
\begin{equation} \label{eq:feinn_continuous_loss}
  u_\mathcal{N} \in 
  \argmin{w_\mathcal{N} \in \mathcal{N}}  
  \mathscr{L}(w_\mathcal{N}), \qquad \mathscr{L}(w_\mathcal{N}) \doteq \norm{\mathcal{R}_h({\pi}_h(w_\mathcal{N}))}_{X'},
\end{equation} 
where $u_{\mathcal{N}}$ is the \ac{nn} approximation of the solution, and $\norm{\cdot}_{X'}$ is some suitable dual norm. Different norms will be considered in the numerical experiments in \sect{sec:experiments}.
However, we only consider $X = U$ in the numerical analysis. 

In general, it is unclear whether $\mathcal{N}$ and $\mathcal{N}_h$ are closed spaces. If $\mathcal{N}_h$ is not closed, we cannot attain the minimum in the optimisation problem \eqref{eq:feinn_continuous_loss}, since we work on $\mathcal{N}_h$, not its closure. To be mathematically precise, the \ac{feinn} problem must be stated as: find
\begin{equation} \label{eq:feinn_inf_continuous_loss}
  u_{\mathcal{N}} \in \mathcal{N} \ : \ \pi_h(u_{\mathcal{N}}) \in 
  \argmin{w_{\mathcal{N},h} \in \overline{\mathcal{N}_h}} \norm{\mathcal{R}_h(w_{\mathcal{N},h})}_{X'}.
\end{equation}  
However, for any tolerance $\epsilon > 0$, we can find 
\begin{equation}\label{eq:feinn_quasimin}
  u_{\mathcal{N},\epsilon} \in \mathcal{N} \ : \ 
  \norm{\mathcal{R}_h(\pi_h(u_{\mathcal{N},\epsilon}))}_{X'} \leq 
  \argmin{w_{\mathcal{N},h} \in \overline{\mathcal{N}_h}}
  \norm{\mathcal{R}_h(w_{\mathcal{N},h})}_{X'} + \epsilon.
\end{equation} 
\begin{remark}
We refer to the solution of (\ref{eq:feinn_quasimin}) as an $\epsilon$-\emph{quasiminimiser} of the \ac{feinn} problem.  The concept of quasiminimisers in the frame of \acp{pinn} has already been  previously discussed~\cite{Shin2023,Brevis2022}.  This is of relative practical relevance since the problem is solved iteratively up to a certain level of accuracy and one can produce a sequence of $\epsilon$-\emph{quasiminimisers} that converges to one of the solutions of the \ac{feinn} problem (\ref{eq:feinn_inf_continuous_loss}).
\end{remark}

\begin{remark}
It has been proven in \cite{Petersen2020} that the space of functions that can be generated by a \ac{nn} is not closed in $L^p$-norms for $0 < p < \infty$ for most activation functions. For the $L^\infty$-norm, and (parameterised) ReLU activation functions, the \ac{nn} space is proven to be closed in some cases. ReLU activation functions are not smooth enough for \acp{pinn} or Deep Galerkin methods \cite{Magueresse2024}. However, in \acp{feinn}, $\mathcal{N}_h$ only depends on the pointwise values of the neural network on the nodes of $U_h$ and ReLU activation functions are suitable, since the method computes spatial derivatives of \ac{fe} functions. Thus, using the results in \cite[Sec. 3.2]{Petersen2020}, one can check that $\mathcal{N}_h$ is closed for ReLU activation functions and the minimum in (\ref{eq:feinn_inf_continuous_loss}) is attained. This is not the case for other activation functions (see Ex.~\ref{ex:tanh}).  
\end{remark}

As discussed in~\cite{Badia2024}, we can define a discrete Riesz projector $\mathcal{B}_h^{-1}: V_h' \to V_h$ such that
\begin{equation*}
    \mathcal{B}_h^{-1}\mathcal{R}_h(w_h) \in V_h \ : \left( \mathcal{B}_h^{-1}\mathcal{R}_h(w_h), v_h \right)_U  = \mathcal{R}_h(w_h)(v_h), \quad \forall v_h \in V_h,
\end{equation*}
and $\mathcal{B}_h^{-1}$  is (an approximation of) the inverse of the corresponding Gramm matrix in $V_h$. Then, we can rewrite the loss function in primal norm as follows: 
\begin{equation} \label{eq:feinn_precond_loss}
  \mathscr{L}(u_\mathcal{N}) = \norm{\mathcal{B}_h^{-1}\mathcal{R}_h ( {\pi}_h(u_\mathcal{N}))}_X.
\end{equation}
Unlike loss functions that rely on discrete norms of the residual vector, this loss function, \emph{preconditioned} by $\mathcal{B}_h$, is well-defined in the limit $h \downarrow 0$. In practice, $\mathcal{B}_h^{-1}$ can be replaced by any spectrally equivalent approximation to reduce computational costs; e.g. one \ac{gmg} cycle has been proposed in \cite{Badia2024}.

\begin{remark}
We note that the \ac{fe} residual $\mathcal{R}_h$ in~\eqref{eq:feinn_continuous_loss} is isomorphic to the vector $[\mathbf{r}_h(w_h)]_i = \left< \mathcal{R}(w_h), \varphi^i \right> \doteq \mathcal{R}(w_h)(\varphi^i)$, where $\{\varphi^i\}_{i=1}^N$ are the \ac{fe} shape functions that span the test space $V_h$. As a result, we can use the following loss function:
\begin{equation} \label{eq:feinn_discrete_loss}
    \mathscr{L}(u_\mathcal{N}) = \norm{\mathbf{r}_h(\pi_h(u_\mathcal{N}))}_\chi,
\end{equation}
where $\norm{\cdot}_\chi$ is an algebraic norm of a vector. The algebraic norm is ill-posed at the continuous level and, as a consequence, the convergence of the minimisation process is expected to deteriorate as $h \downarrow 0$~\cite{Mardal2010}.
\end{remark}

\subsection{$h$-Adaptive finite element interpolated neural networks} \label{subsec:method-afeinn}
The $h$-adaptive \ac{feinn} method is essentially the \ac{feinn} method equipped with a dynamically/automatically adapting mesh. The method is sketched in Alg.~\ref{alg:afeinn}. It consists of a loop of the form:
\begin{equation*}
  \text{TRAIN} \rightarrow \text{ESTIMATE}  \rightarrow \text{MARK} \rightarrow \text{ADAPT}.
\end{equation*}
At the beginning of the $i$-th loop iteration, the \ac{nn} realisation $u_{\mathcal{N}(\pmb{\theta}_i)}$ has initial parameters $\pmb{\theta}_i$. After training of $u_{\mathcal{N}(\pmb{\theta}_i)}$ using the weak \ac{pde} loss built upon the partition $\mathcal{T}_i$, the parameters of the \ac{nn} become $\pmb{\theta}_{i+1}$.
Then, we estimate the local error for each element $K \in \mathcal{T}_i$ using an error indicator $\zeta_K$. Next, we mark a set of elements $\mathcal{M}_i^{\rm r}$ for refinement and another set $\mathcal{M}_i^{\rm c}$ for coarsening. We denote the refinement ratio as $\delta^{\rm r} \in (0, 1)$, and the coarsening ratio as $\delta^{\rm c} \in (0, 1)$.
To determine $\mathcal{M}_i^{\rm r}$ and $\mathcal{M}_i^{\rm c}$, we use D\"{o}rfler marking. Essentially, we select the top $\delta^{\rm r}$ elements with the largest estimator for refinement and the bottom $\delta^{\rm c}$ with the smallest estimator for coarsening. We then adapt the marked elements to construct a new  partition $\mathcal{T}_{i+1}$. Finally, we start a new iteration of the loop by interpolating $u_{\mathcal{N}(\pmb{\theta}_{i+1})}$ onto $\mathcal{T}_{i+1}$. This iterative procedure continues until a stopping criterion is met. 

\begin{algorithm}[h]
  \caption{The $h$-adaptive \ac{feinn} procedure}\label{alg:afeinn}
  \KwSty{Requirements:} $k_U$, stopping criterion, $\delta^{\rm r} \in (0, 1)$, $\delta^{\rm c} \in (0, 1)$\\
  \KwIn{$u_{\mathcal{N}(\pmb{\theta}_0)}$, $\mathcal{T}_0$}
  \KwOut{$u_{\mathcal{N}(\pmb{\theta}_{I+1})}$, $\mathcal{T}_I$}
  \For{$i = 0, 1, 2, ...$}{
    TRAIN $u_{\mathcal{N}(\pmb{\theta}_i)}$ by interpolating it into a \ac{fe} space of order $k_U$ built out of 
    $\mathcal{T}_i$ using the loss defined in~\eqref{eq:feinn_discrete_loss} (or~\eqref{eq:feinn_precond_loss} for the preconditioned case); update \ac{nn} parameters from $\pmb{\theta}_i$ to $\pmb{\theta}_{i+1}$\;
    \If{stopping criterion is met}{\KwSty{break}}
    ESTIMATE the local error for each element $K \in \mathcal{T}_i$ by $\zeta_K = \zeta_K(u_{\mathcal{N}(\pmb{\theta}_{i+1})})$\;
    MARK a set $\mathcal{M}_i^{\rm r} \subset \mathcal{T}_i$ with largest $\zeta_K$ such that $\#\mathcal{M}_i^{\rm r} \approx \delta^{\rm r}\#\mathcal{T}_i$ and another set $\mathcal{M}_i^{\rm c} \subset \mathcal{T}_i$ with smallest $\zeta_K$ such that $\#\mathcal{M}_i^{\rm c} \approx \delta^{\rm c}\#\mathcal{T}_i$\;
    ADAPT the mesh: refine $K \in \mathcal{M}_i^{\rm r}$, and coarsen $K \in \mathcal{M}_i^{\rm c}$ to construct a new partition of the domain $\mathcal{T}_{i+1}$\;
  }
  $I \gets i$\;
  \KwRet{$u_{\mathcal{N}(\pmb{\theta}_{I+1})}$, $\mathcal{T}_I$}
\end{algorithm}

Alg.~\ref{alg:afeinn}  can be in principle combined with any kind of adaptive mesh approach. Mainly because of performance and scalability reasons, in this work we use hierarchically-adapted (i.e., nested) non-conforming octree-based meshes.  
We provide more details on our particular implementation of Alg.~\ref{alg:afeinn} in \sect{subsec:nonconforming_interpolation}.

\subsection{Gradient-conforming discretisation}

As a model problem in this work, we consider the Poisson equation. The problem reads as (\ref{eq:abstract-weak_form}) with $\tilde{U} = H^1(\Omega)$, $U = H^1_0(\Omega)$ and the bilinear form and functional defined as follows:
\begin{equation*}
    a(u,v) = \int_{\Omega} \pmb{\nabla}(u) \cdot \pmb{\nabla}(v), \quad 
\ell(v) = \int_{\Omega} f v,
\end{equation*}

We consider an \emph{exotic} Petrov-Galerkin \ac{fe} method. The well-posedness of the discretisation is guaranteed by the discrete inf-sup condition (\ref{eq:discrete-infsup}) and the same dimension for trial and test spaces  (see~\cite{Griffiths1978}). We use a grad-conforming nodal Lagrangian \ac{fe} space $U_h$ of order $k_U$ built upon a (possibly non-conforming) partition $\mathcal{T}_h$ of $\Omega$. The \ac{fe} space $U_h$ is built upon the partition $\mathcal{T}_h$. We use a \emph{linearised} test \ac{fe} space $V_h$ ($k_V=1$), built upon the partition $\mathcal{T}_h$ obtained after $k_U$ levels of uniform refinement applied to $\mathcal{T}_h$. It is easy to check that the dimension of ${U}_h$ and $V_h$ are identical.
The reason behind this choice of the test space is the dramatic improvement in the convergence of the optimisation process compared to a Galerkin discretisation, as observed, e.g., in~\cite{Berrone2022,Badia2024}.

\subsection{Error indicators} \label{subsec:method-errorind}
We introduce three types of error indicators to estimate the local errors of each element. In this section, we restrict the exposition to the Laplacian problem for the sake of simplicity. However, the error indicators can be extended to other \acp{pde}. 

The first error indicator is the classic Kelly error estimator~\cite{Kelly1983,Ainsworth1997}, a well-established indicator in $h$-adaptive \ac{fem}, particularly for the Poisson equation. This error indicator estimates the error per element by integrating the gradient jump of the \ac{nn} interpolation along the faces of each element. Specifically, for the element $K$, the Kelly error indicator is defined as
\begin{equation}
    \zeta_K^k = \sqrt{\sum_{F \in \partial K} c_F \int_{F} \left\llbracket \frac{\partial ({\pi}_h(u_{\mathcal{N}}) + \bar{u}_h)}{\partial n} \right\rrbracket^2}, \label{eq:kelly_error}
\end{equation}
where $c_F$ is a scale factor, $\partial K$ is the boundary of $K$, and $\llbracket \cdot \rrbracket$ denotes the jump across the face $F$. 
The \ac{fe} offset $\bar{u}_h$ term in $\zeta_K^k$ can be excluded if the Dirichlet boundary is exactly approximated by the employed finite elements, i.e. $\bar{u}_h = g$.
In the notation $\zeta_K^k$, the superscript $k$ stands for ``Kelly'', and the subscript $K$ denotes an element of the mesh.  We use $c_F = h_K / 24$ as proposed in~\cite{Ainsworth1997}, where $h_K$ is the characteristic size of element $K$. 

Then we introduce another error indicator that leverages the \ac{nn} and the strong form of the \ac{pde}. The error estimator at element $K$ is defined as
\begin{equation} \label{eq:nn_error}
    \zeta_K^n = \sqrt{\int_{K} |\pmb{\Delta} u_{\mathcal{N}} + f|^2}.
\end{equation}
The superscript $n$ in $\zeta_K^n$ stands for ``network''. 
We note that $\zeta_K^n$ is similar to the mean \ac{pde} residual $\varepsilon_r$ defined in~\cite{Lu2021}. 
The main difference is that $\zeta_K^n$ is evaluated with a high-degree Gauss quadrature, while $\varepsilon_r$ is computed by Monte Carlo integration. At each mesh adaptation step, the strategy in~\cite{Lu2021} adds more collocation points only on the region with the highest $\varepsilon_r$. Instead, our approach can refine multiple elements with the largest errors and coarsen those with the smallest errors simultaneously.

Finally, we include the real error of the interpolated \ac{nn} as the third indicator, which is defined for element $K$ as:
\begin{equation*}
    \zeta_K^r = \sqrt{\int_{K} |{\pi}_h(u_{\mathcal{N}}) + \bar{u}_h - u|^2},
\end{equation*}
where $u$ is the exact solution. The superscript $r$ in $\zeta_K^r$ stands for ``real''. In practice, the real error is unavailable. In this study, we use it as a reference to assess and compare the performance of the other two indicators.

\section{Numerical analysis}\label{sec:analysis}

In this section, we extend the analysis in~\cite{Badia2024} in multiple ways.
In particular, in Sec.~\ref{sec:error_analysis}, we bound the error of the interpolated \ac{nn}
under the general scenario in which the \ac{nn} cannot exactly emulate the \ac{fe} space $U_h$, i.e., $\mathcal{N}_h \subsetneq U_h$ (or, equivalently, the interpolation operator $\pi_h: \mathcal{N} \mapsto U_h$ is not surjective).
Remarkably, we prove that the error of the interpolated \ac{nn} is bounded by the interpolation error among all possible vector spaces that result from the interpolation of the last-layer functions in the \ac{nn} architecture. In Sec.~\ref{sec:quasi_emulation}, we perform error analysis under a weakened assumption of exact emulation which we refer to as {\em quasi-emulation}. This analysis is relevant as exact emulation might not be feasible for some activation functions, regardless of the expressivity of the \ac{nn}. Finally, in Sec.~\ref{subsec:generalisation-error}, we also discuss the well-posedness of the \ac{feinn} problem in a quotient space and discuss the accuracy of the non-interpolated \ac{nn} by considering the maximum distance between elements of an equivalence class. 

\subsection{Error analysis} \label{sec:error_analysis}

The error analysis in this section relies on the discrete inf-sup condition \eqref{eq:discrete-infsup}, and  can readily be applied to indefinite systems. In the following theorem, we state the well-posedness of the residual minimisation problem in a \ac{fe} setting.

\begin{theorem}\label{th:fe-well-posed}
  Let us consider a pair of \ac{fe} spaces $U_h$ and $V_h$ that satisfy the inf-sup condition (\ref{eq:discrete-infsup}).
  The discrete quadratic minimisation problem 
  \begin{equation}\label{eq:fe-min}
  u_h = \underset{w_h \in U_h}{\mathrm{arg} \, \mathrm{min}} \| r_h(w_h)\|_U, \qquad r_h(w_h) \doteq \mathcal{B}_h^{-1} \mathcal{R}_h(w_h),
  \end{equation} 
  is well-posed. The problem can be re-stated as: find $u_h \in U_h$ and $r_h = r_h(u_h) \in V_h$ such that
  \begin{align}
    (r_h,v_h)_U + a(u_h,v_h) &= \bar{\ell}(v_h), \quad && \forall v_h \in V_h, \label{eq:feinn-weak-1}\\
    a(w_h,r_h) &= 0, \quad && \forall w_h \in U_h. \label{eq:feinn-weak-2} 
  \end{align}
  The solution satisfies the following a priori estimates:
  \begin{equation}\label{eq:apriori}
    \|r_h\|_U \leq \underset{w_h \in U_h}{\mathrm{inf}} 4 \gamma \| u - w_h \|_{U}, \qquad \|u - u_h\|_U \leq \left( 1 + \frac{4 \gamma}{\beta} \right) \underset{w_h \in U_h}{\mathrm{inf}} \|u - w_h\|_U.
  \end{equation}
\end{theorem}

\begin{proof}
  First, we note that the loss function is differentiable (see~\cite[Prop. 3.1]{Badia2024}). To obtain (\ref{eq:feinn-weak-1})-(\ref{eq:feinn-weak-2}), we compute the variational derivative of (\ref{eq:fe-min}) to variations of $u_h$ and recall the definition of the discrete Riesz projector. Under the assumption of the discrete inf-sup condition (\ref{eq:discrete-infsup}), we can prove the well-posedness of the problem using the Babuska-Brezzi theory. Namely, using (\ref{eq:discrete-infsup}), we can readily find $k_h \in V_h$ such that $\| k_h \|_U = \| u_h \|_U$  and $a(u_h,k_h) \geq \beta \| u_h \|_U^2$. We can take $v_h = r_h + \alpha k_h$ 
  for an arbitrary $\alpha > 0$ and $w_h = -u_h$, to get
  \begin{align}
    \|r_h\|^2_U + \alpha a(u_h,k_h) + \alpha (r_h, k_h )_U & \geq  \|r_h\|^2_U + \alpha \beta \|u_h\|^2_U - \frac{\alpha \beta}{ 2} \|u_h\|_U^2 - \frac{\alpha}{2 \beta} \|r_h\|^2_{U} \\ & = \left( 1 - \frac{\alpha}{2 \beta} \right) \| r_h \|^2_U + \frac{\alpha \beta}{2} \| u_h \|^2_U. 
  \end{align}
  Taking $\alpha = \beta$, we get, using (\ref{eq:feinn-weak-1}):
  \begin{align}\label{eq:apriori-aux}
    \frac{1}{4}(\|r_h \|_U + \beta \| u_h \|_U)^2 & \leq 
    \frac{1}{2} \| r_h \|^2_U + \frac{\beta^2}{2} \| u_h \|^2_U \leq \|r_h\|^2_U + \beta a(u_h,k_h) + \beta (r_h, k_h )_U \\ & \leq \|\bar{\ell}\|_{U'} (\|r_h \|_U + \beta \| u_h \|_U).
  \end{align}
  Using the fact that $\bar{\ell}(v_h) = a(u,v_h)$, and subtracting $a(w_h,v_h)$ on both sides of (\ref{eq:feinn-weak-1}) for an arbitrary $w_h \in U_h$ and using the continuity of the bilinear form $a$, we get the following error estimate from (\ref{eq:apriori-aux}):
\begin{equation}\label{eq:error}
  \| r_h \|_U + \beta \|u_h - w_h\|_U \leq 4 \gamma \|u - w_h\|_U.
\end{equation}
Using the triangle inequality, we readily get the second estimate in (\ref{eq:apriori}). It proves the theorem.
\end{proof}
\begin{remark}
We note that $dim(U_h) \neq dim(V_h)$ in general, and the trial space can be of lower dimension than the test space. In the case in which $dim(U_h) = dim(V_h)$, the solution $u_h$ of (\ref{eq:feinn-weak-1})-(\ref{eq:feinn-weak-2}) is the solution of (\ref{eq:abstract-weak_form}) and $r_h = 0$.  
\end{remark}  

The previous theorem readily applies to any subspace $W_h \subseteq U_h$, since the inf-sup condition readily holds for $W_h \times V_h$. We use the notation $u(W_h)$ and $r(W_h)$ to denote the unique solution of (\ref{eq:feinn-weak-1})-(\ref{eq:feinn-weak-2}) on $W_h \times V_h$. We can prove the following result.

\begin{proposition}\label{prop:aux-r-u-bound}
  Let us consider $u_{h,1}, \, u_{h,2} \in U_h$.
  The following error estimate holds: 
  \begin{equation}\label{eq:restricted}
    \frac{1}{\beta} \| u_{h,1} - u_{h,2} \|_U \leq \|r(u_{h,1}) - r(u_{h,2})\|_U \leq \gamma \|u_{h,1} - u_{h,2} \|_U. 
  \end{equation}
\end{proposition}
\begin{proof}
  Let us denote for conciseness $u_h \doteq u_{h,1}$, $u_h' \doteq u_{h,2}$, $r_h \doteq r(u_{h,1})$, $r_h' \doteq r(u_{h,2})$. Using the definition of $r_h(\cdot)$ in  (\ref{eq:fe-min}), we have:
  \begin{align*}
    (r_h, v_h)_U &= - a(u_h, v_h) + \bar{\ell}(v_h), \\
    (r_h', v_h)_U &= - a(u_h', v_h) + \bar{\ell}(v_h), \qquad \forall v_h \in V_h.
  \end{align*}
  Subtracting the previous equations, we get
  \begin{equation}
    (r_h - r_h', v_h)_U = -a(u_h - u_h', v_h), \qquad \forall v_h \in V_h. 
  \end{equation}
  The upper bound can readily be obtained using the continuity of $a$ in (\ref{eq:continuity}) , while the lower bound is obtained using the discrete inf-sup condition (\ref{eq:discrete-infsup}) and the Cauchy-Schwarz inequality.
\end{proof}

\begin{proposition}\label{prop:aux-two-spaces}
  Let us consider $y_h \in U_h$ and $
  W_h \subseteq U_h$ such that $\|r(y_h)\|_U \leq \|r_h(W_h)\|_U + \epsilon$, $\epsilon \geq  0$.  The following a priori error estimate holds:
  \begin{equation}
    \| u - y_h \|_U \leq  \rho \inf_{w_h \in W_{h}} \| u - w_{h} \|_U + \beta \epsilon, \qquad \rho \doteq 1 + \frac{4 \gamma}{\beta} + {8 \gamma}{\beta}.
  \end{equation}
  \end{proposition}
  \begin{proof}
  Using Prop. \ref{prop:aux-r-u-bound}, the triangle inequality, the statement of the proposition and the error estimate in (\ref{eq:apriori}), we get:
  \begin{equation*}
    \frac{1}{\beta} \| y_h - u(W_h) \|_U \leq \|r(y_h) - r(W_h)\|_U \leq 2 \| r(W_h) \|_U + \epsilon \leq 8 \gamma \inf_{w_{h} \in W_{h}} \| u - w_{h} \|_U + \epsilon.
  \end{equation*}
  Then, using the triangle inequality and (\ref{eq:apriori}), we get:  
  \begin{equation*}
    \| u - y_h \|_U \leq \|_U u - u(W_h) \|_U + \| y_h - u(W_h) \|_U \leq \rho \inf_{w_h \in W_{h}} \| u - w_{h} \|_U + \epsilon. 
  \end{equation*}
  \end{proof}

Next, we prove that the infimum (\ref{eq:feinn_inf_continuous_loss}) (or in (\ref{eq:feinn_quasimin})) is the same if we restrict to the space of solutions of (\ref{eq:feinn-weak-1})-(\ref{eq:feinn-weak-2}) for all last-layer interpolated spaces, i.e., $\mathcal{M}_h \doteq \left\{ u_h(S_h(\pmb{\theta}_{hl})) : \pmb{\theta}_{hl} \in \mathbb{R}^{n_{hl}}\right\}$. 

\begin{proposition}\label{prop:eq-min}
The following equalities hold:
\begin{equation}
 \underset{w_\mathcal{N} \in \mathcal{N}}{\mathrm{inf}} \| r_h(\pi_h(w_\mathcal{N}))\|_U =
 \underset{w_h \in \mathcal{N}_h}{\mathrm{inf}} \| r_h(w_h)\|_U = 
 \underset{u_h \in \mathcal{M}_h}{\mathrm{inf}} \| r_h(u_h)\|_U.
\end{equation} 
\end{proposition}
\begin{proof}
The first equality is an obvious result of the definition for $\mathcal{N}_h$. The second equality uses the fact that $\mathcal{N}_h$ can be represented as the union of all last-layer spaces. Due to Th. \ref{th:fe-well-posed}, among all functions in $S_h$, there exists a unique solution $u_h(S_h)$ that minimises the functional. It is the solution of (\ref{eq:feinn-weak-1})-(\ref{eq:feinn-weak-2}) for $S_h \times V_h$ and its residual is denoted by $r_h(S_h)$. It proves the result.
\end{proof}
\begin{remark}
  To consider the minimisation on $\mathcal{M}_h$ is practically possible. It involves to solve a linear system for the last-layer weights of the \ac{nn}, e.g., using an iterative Krylov method. It motivates the usage of hybrid solvers that combine non-convex optimisation (e.g., Adam) and linear solvers. If not at every epoch, a final linear solve can be performed to obtain the optimal solution in the last-layer space.
\end{remark}

Let us consider the closure $\overline{\mathcal{N}_h}$ of $\mathcal{N}_h$ in $U$. $U_h$ is closed as it is a finite-dimensional vector space on a complete normed field ($\mathbb{R}$ or $\mathbb{C}$). Thus, $\overline{\mathcal{N}_h} \subseteq U_h$. We can prove the following result.

\begin{theorem}\label{eq:feinn-error}
  For any $\epsilon > 0$ and $\epsilon$-quasiminimiser $u_{\mathcal{N},\epsilon}$ of (\ref{eq:feinn_quasimin}), the following a priori error estimate holds:
  \begin{equation}
\| u - \pi_h(u_{\mathcal{N},\epsilon}) \|_U 
 \leq \rho \inf_{w_h \in \mathcal{N}_h} \| u - w_h \|_U + \beta \epsilon.
  \end{equation}
\end{theorem}
\begin{proof}
Let us consider the minimisation problem on $\overline{\mathcal{N}_h}$. The infimum is attained since the space is closed. 

Thanks to Prop. \ref{prop:eq-min}, the minimum in (\ref{eq:feinn_inf_continuous_loss}) is attained for some $u_h \in \overline{\mathcal{M}_h}$.  
Thus, for any  $\epsilon$-quasiminimiser $u_{h,\epsilon}$ of (\ref{eq:feinn_quasimin}), it holds $\|r_h(u_{h,\epsilon})\|_U \leq \|r_h(u_h)\|_U + \epsilon \leq \|r_h(v_h)\|_U + \epsilon$ for any $v_h \in \mathcal{M}_h$.
Invoking Prop.~\ref{prop:aux-two-spaces} and (\ref{eq:nn_manifold-span}), we prove the result.
\end{proof}

\subsection{Quasi-emulation} \label{sec:quasi_emulation}

In the previous section, we have considered the case in which the \ac{nn} space cannot emulate the \ac{fe} space, i.e., $\mathcal{N}_h \subsetneq U_h$. If the \ac{nn} is expressive enough, we can prove that $\mathcal{N}_h = U_h$ following the construction in~\cite[Prop. 3.2]{Badia2024}.

\begin{proposition}\label{prop:emulation-1}
  Let us consider an adaptive Lagrangian \ac{fe} space $U_h$ in $\mathbb{R}^d$  with $N$ nodes, which can be defined on a non-conforming mesh and include hanging nodes. Let $\mathcal{N}$ be a neural network architecture with ReLU activation function, 3 hidden layers and $(3dN,dN,N)$ neurons per layer. The interpolation operator $\pi_h: \mathcal{N} \mapsto U_h$ is surjective.   
\end{proposition}

\begin{proof}
  The proof follows the same lines as in~\cite[Prop. 3.2]{Badia2024}. The neural network above can represent in the last layer all the shape functions in the adaptive \ac{fe} space $U_h$, both the free ones as the ones that are constrained and associated to hanging nodes. The linear combination of these last layer shape functions can represent any function in $U_h$. As a result, the interpolation operator $\pi_h$ is surjective.   
\end{proof}

\begin{remark}
The previous result is overly pessimistic and can be improved when using meshes that are obtained after adaptive mesh refinement of a regular mesh. In this case, the \ac{fe} space $U_h$ can be represented by a \ac{nn} with a much smaller number of neurons, since one can exploit a tensor product structure in the proof (see~\cite{Badia2024}).
\end{remark}

Exact emulation is not attainable for some activation functions, like $\tanh$, regardless of the expressivity of the \ac{nn}. In this case, we introduce the concept of quasi-emulation.

  \begin{definition}
    A \ac{nn} is a $\delta$-emulator of a vector space $U_h$ if for all $w_h \in U_h$, there exists a $w \in \mathcal{N}$ such that $\| \pi_h(w) - w_h \|_U \leq \delta$. 
  \end{definition}

  \begin{example}\label{ex:tanh}
  Let us consider a 1D problem for simplicity and a fully-connected \ac{nn} as
in \sect{subsec:method-nn}
    with $\rho=\tanh$ activation, $m \in \mathbb{N}$ and the weights and biases:
  \begin{equation*}
      \pmb{W}_1 = \begin{bmatrix} m, \ -m \end{bmatrix}^T, \ \pmb{b}_1 = \begin{bmatrix} \frac{m}{4}, \ \frac{m}{4} \end{bmatrix}^T, \quad \pmb{W}_2 = \begin{bmatrix} m, \ m \end{bmatrix}, \ \pmb{b}_2 = \begin{bmatrix} 0 \end{bmatrix}.
  \end{equation*}
  One can check that the corresponding realisation $f_m({x}) = \rho \circ (\pmb{W}_2 (\rho \circ (\pmb{W}_1 {x} + \pmb{b}_1) + \pmb{b}_2))$ satisfies $f_m({x} = 0) = 1 - \delta$, $f_m({x} = -1) = f_m({x} = 1) = c\delta$  for some constant $c$, and monotonically decreases as $x$ goes to $\pm \infty$. Thus, $\lim_{m \to \infty} f_m(x)= h(x)$, where $h(x)$ is the hat function with support in $[-1,1]$ and value 1 at $x=0$. We can readily obtain a hat function in any interval by applying translation and/or scaling. 
  Thus, we can construct a \ac{nn} with a set of last-layer functions that emulates a \ac{fe} space basis up to a desired tolerance $\delta > 0$.
  Similar constructions are possible in higher dimensions.
\end{example}

Due to the well-posedness of the \ac{fe} problem in Th. \ref{th:fe-well-posed}, there is a unique minimum $u_h \in U_h = \overline{\mathcal{N}_h}$ of (\ref{eq:feinn-weak-1})-(\ref{eq:feinn-weak-2}). For exact emulation, $\mathcal{N}_h = \overline{\mathcal{N}_h}$. Thus, we can attain minimisers 
\begin{equation*} 
    u_{\mathcal{N}} \in \mathcal{N} \ : \ \pi_h(u_{\mathcal{N}}) \in 
    \argmin{w_{\mathcal{N},h} \in {\mathcal{N}}_h} \norm{\mathcal{R}_h(w_{\mathcal{N},h})}_{X'}.
\end{equation*}
In fact, the minimisers are all the $u_\mathcal{N} \in \mathcal{N}$ such that $\pi_h(u_\mathcal{N}) = u_h$ and the a priori error estimate (\ref{eq:apriori}) holds. For quasi-emulation, the following result is a direct consequence of Th. \ref{th:fe-well-posed} and the previous definition.
\begin{corollary}
  Let us consider a \ac{nn} $\mathcal{N}$ that is a $\delta$-emulator of the \ac{fe} space $U_h$. For any $\epsilon > 0$, an $\epsilon$-quasiminimiser $u_\mathcal{N}$ of (\ref{eq:feinn_quasimin}) satisfies the following a priori error estimate:
  \begin{equation}
\| u - \pi_h(u_{\mathcal{N}}) \|_U 
\leq \rho \inf_{w_h \in U_h} \| u - w_h \|_U + \beta \epsilon. 
  \end{equation}
\end{corollary}

\subsection{Generalisation error and equivalence classes}\label{subsec:generalisation-error}

In the discussions above, we have only bounded the error of the interpolated \ac{nn}. Analogously, we have discussed the well-posedness in the space of interpolations of the \ac{nn}. The total error of the \ac{nn} is harder to bound, since it is \emph{transparent} to the loss.

In order to better bound the total error of the \ac{nn}, we can make use of quotient spaces (see a related approach in~\cite{Rojas2023} for \acp{vpinn}). Let us define the following equivalence relation $\sim$: two \ac{nn} realisations $\mathcal{N}(\pmb{\theta})$ and $\mathcal{N}(\pmb{\theta}')$ are equivalent if their interpolations on the adaptive \ac{fe} space $U_h$ are identical, i.e., $\pi_h(\mathcal{N}(\pmb{\theta})) = \pi_h(\mathcal{N}(\pmb{\theta}'))$. The equivalence class of $\mathcal{N}(\pmb{\theta})$ is denoted by $[\mathcal{N}(\pmb{\theta})]$. We can define the quotient space $\mathcal{N}/\sim$ as the set of all equivalence classes $[\mathcal{N}(\pmb{\theta})]$. 

Let us consider the case of exact emulation for simplicity. In this situation, the problem (\ref{eq:feinn_inf_continuous_loss}) admits a unique solution $u_h$ in $U_h$. Then, the problem is also well-posed in the quotient space $\mathcal{N} / \sim$; the unique solution is the equivalence class $[u_h] \in \mathcal{N} / \sim$. Analogously, in the case of quasi-emulation, quasi-minimisers in the quotient space will be arbitrarily close to $[u_{h,\epsilon}]$.

Assuming that the training process converges to the global minimiser, we can prove the following result.

\begin{proposition}\label{prop:accuracy}
  Let $u_h \in U_h$ be the solution of (\ref{eq:diffusion_fe_weak_form}). Let us assume that the training process has attained an $\epsilon$-quasiminimiser $u_{h,\epsilon} \in \mathcal{N}_h$, or analogously, a $[u_{h,\epsilon}] \in \mathcal{N} / \sim$. Any possible realisation $u_{\mathcal{N}} \in [u_{h,\epsilon}] \subset \mathcal{N}$ obtained after training satisfies the following a priori error result:
  \begin{equation}
    \norm{u - u_{\mathcal{N}}}_U \leq \norm{u - u_h}_U + \sup_{v \in [\pi_h(u_{\mathcal{N}})]} \norm{\pi_h^\perp(v)}_U + \epsilon,
  \end{equation}
  where $\pi_h^\perp(u) = u - \pi_h(u)$.
\end{proposition}

\begin{proof}
  Using the triangle inequality and taking the supremum of the error between the \ac{nn} and the \ac{fe} solution in the \ac{nn} architecture, we prove the result.
\end{proof} 

The first term in the error estimate is the \ac{fe} error, the second term is the generalisation error, and the third term is a training error. The generalisation error describes the maximum distance between the \ac{fe} solution and the \ac{nn} realisations in the corresponding equivalence class. Since the loss function does not \emph{feel} $\|\pi_h^\perp(u_{\mathcal{N}})\|$, there is no explicit control on it. This error is related to the expressiveness of the \ac{nn} with respect to the \ac{fe} mesh. The more expressive the \ac{nn}, the larger the generalisation error can be, as it is well-known in the machine learning literature (over-fitting). Besides, if the \ac{nn} is too coarse, the emulation assumption will not hold and the \ac{feinn} will not be able to capture the \ac{fe} solution. 
\footnote{
This discussion motivates the usage of a regularisation term that controls the distance between the \ac{nn} and the \ac{fe} solution. For example, one could consider a term of the order $ h^{-2}\| \pi^\perp (u_{\mathcal{N}})\|_{L^2(\Omega)}$ or $ \| \pi^\perp (u_{\mathcal{N}})\|_{H^1(\Omega)}$, integrated with a high order quadrature. Such a term is weakly consistent, i.e., it does not affect the convergence of \acp{feinn} compared to \ac{fem}. This stabilisation resembles the so-called orthogonal subscale stabilisation common in the \ac{fe} literature to circumvent the inf-sup condition (see, e.g.,~\cite{Codina2017,Badia2012}). We have not considered stabilisation in this work.
}

The previous analysis is \ac{fe}-centric, in the sense that we consider the \ac{nn} error by comparing it with the \ac{fe} solution and use \ac{fe} theory to quantify the error. Even though the \ac{feinn} can capture the \ac{fe} solution with a much smaller number of parameters than the number of \acp{dof} in the \ac{fe} space~\cite{Badia2024}, it can be argued that if the only aim is to attain the \ac{fe} error, it is more effective to solve the \ac{fe} problem. However, the \ac{feinn} approach can naturally deal with inverse problems by simply adding data misfit terms to the loss function and the \ac{nn} parametrisation regularises the inverse problem and can readily be used in a one-shot fashion, avoiding the need of adjoint problems~\cite{Badia2024}. It has also been observed experimentally in~\cite{Badia2024} that the \ac{fe}-centric analysis of \acp{feinn} is overly pessimistic in general. The trained \ac{nn} using the \ac{feinn} approach can be up to two orders of magnitude more accurate than the \ac{fe} solution for forward problems. In \sect{sec:experiments}, we analyse this fact in the case of singular solution and sharp gradients combined with $h$-adaptivity and a posteriori error estimation. A mathematical justification of the superior accuracy of the \ac{feinn} trained \ac{nn} compared to the \ac{fem} solution is still open.

The two error estimates in Sec. \ref{subsec:method-errorind} are related to these two different perspectives. The a posteriori error estimate \eqref{eq:kelly_error} aims to reduce the error of the interpolated \ac{nn}, following the \ac{fe}-centric perspective. Instead, the estimate \eqref{eq:nn_error} aims to directly reduce the error between the \ac{nn} and the exact solution, considering the interpolation just as a way to enforce the Dirichlet condition, define a trial space to compute the weak residual and perform an approximate numerical integration.

\section{Implementation} \label{sec:implementation}

\subsection{Automatic mesh adaptation using forest-of-octrees} \label{subsec:auto_mesh_adapt}

Automatic mesh adaptation is a key component of $h$-adaptive \acp{feinn}.
In this work, we leverage a particular kind of hierarchically-adapted meshes known as 
forest-of-octrees (see, e.g.,~\cite{Burstedde2011,Badia2020}). 
Forest-of-octrees meshes can be seen as a two-level
decomposition of $\Omega$, referred to as macro and micro level, respectively. The macro level is a suitable {\em conforming} partition $\mathcal{C}$ of $\Omega$ into quadrilateral ($d=2$) or hexahedral cells ($d=3$). This mesh, which may be generated, e.g., using an unstructured mesh generator,  is referred to as the coarse mesh. At the micro level, each of the cells of $\mathcal{C}$ becomes the root of an adaptive octree with cells that can be recursively and dynamically refined or coarsened using the so-called $1:2^d$ uniform partition rule. If a cell is marked for refinement, then it is split into $2^d$ children cells by subdividing all parent cell edges. If all children cells of a parent cell are marked for coarsening, then they are collapsed into the parent cell. The union of all leaf cells in this hierarchy forms the decomposition of the domain at the micro level, i.e., $\mathcal{T}$.

In this work, we utilise the \texttt{GridapP4est.jl}~\cite{GridapP4est2024} Julia package in order to handle such kind of meshes.  
This package, built upon the \texttt{p4est} meshing engine~\cite{Burstedde2011}, is endowed with the so-called Morton space-filling curves, and it provides high-performance and low-memory footprint algorithms to handle forest-of-octrees.

\subsection{Non-conforming interpolation of neural networks} \label{subsec:nonconforming_interpolation}

Unlike our previous work~\cite{Badia2024}, which relies on (uniformly-refined) conforming meshes, in this work 
we interpolate the \ac{nn} onto (grad-conforming) \ac{fe} spaces built out of non-conforming forest-of-octrees meshes.
With this kind of meshes, generating a cell-conforming global enumeration of \acp{dof} is not sufficient to impose the required continuity at the interface of cells with different refinement levels. In order to remedy this, it is standard practice to equip the \ac{fe} space with additional linear multi-point constraints at the \acp{dof} lying on hanging vertices, edges and faces. The structure and set up of such constraints is well-established knowledge, and thus not covered here; see, e.g.,~\cite{Badia2020,Olm2019} for further details. In any case, the presence of these linear multi-point constraints in the \ac{fe} space implies that the interpolation of the \ac{nn} at hanging \acp{dof} cannot by given by its evaluation at the corresponding nodes, but rather by a linear combination of the \ac{nn} evaluations at the nodes of the adjacent coarser cell. To this end, \texttt{GridapP4est.jl} provides the required functionality to build and apply these linear multi-point constraints, and to resolve these constraints during \ac{fe} assembly.

\subsection{Finite element interpolated neural networks}
In this section we briefly overview the main building blocks of our Julia implementation of $h$-adaptive \acp{feinn}.
For a more complete coverage of these (including the inverse problem case), we refer to~\cite{Badia2024}. 
At a purely algebraic level, we can rewrite the preconditioned loss~\eqref{eq:feinn_precond_loss} as:
\begin{equation*}
    \mathscr{L}(\pmb{\theta}) = \norm{r_h(\pmb{\theta})}_X,
\end{equation*}
where {$r_h \in V_h$ is the discrete Riesz projection of the \ac{fe} residual}. The vector of \acp{dof} of $r_h$ is $\mathbf{B}^{-1} (\mathbf{A} \mathbf{u}_h - \mathbf{f})$, where $\mathbf{B}$ is the discretisation in the test space $V_h$ of the Gram matrix associated to the inner product in $U$, e.g., the Laplacian matrix for the Poisson problem, $\mathbf{A}$ is the coefficient matrix resulting from discretisation, $\pmb{\theta}$ is the parameters for the \ac{nn} $u_{\mathcal{N}}$, $\mathbf{u}_h$ is the vector of \acp{dof} of $U_h$, and $\mathbf{f}$ is the \ac{rhs} vector. 
We use  the abstract data structures in \texttt{Gridap.jl}~\cite{Gridap2020,Gridap2022} as implemented by \texttt{GridapP4est.jl} to compute $\mathbf{A}$, $\mathbf{f}$, and the \ac{fe} nodal coordinates $\mathbf{X}$ of those 
\acp{dof} which are not subject to linear multi-point constraints. We utilise \texttt{Flux.jl}~\cite{Flux2018,FluxOSS2018} to build $u_{\mathcal{N}}$ and evaluate $u_{\mathcal{N}}$ at $\mathbf{X}$ to compute $\mathbf{u}_h$. 
Without preconditioning, i.e., $\mathbf{B} = \mathbf{I}$, we can minimise the loss $\norm{\mathbf{A}\mathbf{u}_h(\pmb{\theta}) - \mathbf{f}}_\chi$ directly.
During training, we rely on \texttt{Zygote.jl}~\cite{Zygote2018} to compute the gradient of the loss with respect to $\pmb{\theta}$ (see~\cite{Badia2024} for details). We note that in the context of $h$-adaptive \acp{feinn}, we need to recompute $\mathbf{A}$, $\mathbf{f}$, and $\mathbf{u}_h$ at each adaptation step. {Furthermore, depending on the specific problem, these arrays can be stored cell-wise to improve efficiency.}

If a preconditioner is employed, i.e., $\mathbf{B} \neq \mathbf{I}$, we need to perform these two additional steps at each iteration:
\begin{enumerate}
    \item {Compute $r_h$ by solving $\mathbf{B}^{-1}(\mathbf{A}\mathbf{u}_h(\pmb{\theta})-\mathbf{f})$ on $V_h$}.
    \item Compute the preconditioned loss $\mathscr{L}(\pmb{\theta}) = \norm{r_h(\pmb{\theta})}_X$.
\end{enumerate}
The forward and backward passes of the above steps are both handled by \texttt{Gridap.jl}.
It is noteworthy that in step (1), we do not explicitly compute the inverse of $\mathbf{B}$, {but rather solve the linear system $\mathbf{B}\mathbf{r}_h = \mathbf{A} \mathbf{u}_h - \mathbf{f}$ to compute the vector of \acp{dof} $\mathbf{r}_h$ that uniquely determine $r_h \in  V_h$}.
Considering the problem linear, we can reuse the factorisation of $\mathbf{B}$, leading to linear complexity computational cost in the application of $\mathbf{B}^{-1}$ per iteration. Alternatively, as explored in~\cite{Badia2024}, we can use a cheaper preconditioner $\mathbf{B}$, such as the \ac{gmg} preconditioner, which has linear complexity in the number of \acp{dof} of the \ac{fe} space.
Besides, given that $r_h \in V_h$ and $V_h$ is a linearised \ac{fe} space, and the gradient of $r_h$ is piecewise constant, the computation of $\norm{r_h}_X$ is very cheap.
\section{Numerical experiments} \label{sec:experiments}

\tab{tab:hyperparameters} lists the \ac{nn} architectures and training iterations for the different problems tackled in this section. 
For 2D experiments, we adopt the \ac{nn} architecture in~\cite{Badia2024}, namely $L = 5$ layers and $n = 50$ neurons for each hidden layer. 
For 3D experiments, since we work with a more complex problem with more \acp{dof}, we increase the number of layers to $L=10$.
In addition, we use $\rho = \tanh$ as the activation function, the Glorot uniform method~\cite{Glorot2010} as the \ac{nn}  parameter initialisation strategy. We employ the \texttt{BFGS} optimiser in \texttt{Optim.jl}~\cite{Optimjl2018} to train the \acp{nn}, as recommended in~\cite{Badia2024}.
Moreover, in our preliminary experiments, we have observed that the results are generally robust to variations in \ac{nn} initialisations, provided that the initial mesh is sufficiently fine  and the \ac{nn} is adequately trained.
As a result, at each adaptation step, we heuristically set the number of the training iterations as a function of the number of \acp{dof} in the trial \ac{fe} space. In particular, the higher the number of \acp{dof}, the higher the number of training iterations. 
We acknowledge that advanced/automatic optimiser termination conditions could result in better accuracy/performance trade-offs than the ones shown in this section. However, developing such conditions is beyond the scope of this paper.

\begin{table}[h]
    \begin{tabular}{lllll}
        \toprule
        Section number and problem & \# layers & \# neurons & \# BFGS iterations & \ac{dof} milestones\\
        \hline
        \sect{subsec:arc_wavefront} 2D arc wavefront   & 5    & 50 & [500, 1000, 1500] & [5000, 10000]\\
        \sect{subsec:lshape} 2D L-shaped singularity &   5  & 50  &[3000, 4000, 5000] & [10000, 20000]\\
        \sect{subsec:singularities3d} 3D singularities &10 & 50 &[3000, 4000, 5000] & [50000, 100000]\\
        \toprule
    \end{tabular}
    \caption{Example of selected hyperparameter values for the three types of problems tested in this study. The numbers shown correspond to $k_U=4$ assuming no preconditioner. For different values of $k_U$ and/or preconditioner strategy, please see the corresponding subsection. The number of training iterations at each adaptation step are determined by \ac{dof} milestones (last column of the table). For instance, in the 2D arc wavefront problem, we use a fixed number of 500 iterations up to 5,000 \acp{dof}, 1,000 iterations within 5,000-10,000 \acp{dof}, and 1,500 iterations over 10,000 \acp{dof}.} \label{tab:hyperparameters}
\end{table}

Regarding the refinement and coarsening ratios, for simplicity, we just use fixed ratios for all adaptation steps, even though these ratios might be different for each step in practice. 
We use refinement ratio $\delta^{\rm r} = 0.15$ for 2D cases and $\delta^{\rm r} = 0.1$ for 3D cases. For both space dimensions, the coarsening ratio is chosen to be $\delta^{\rm c} = 0.01$. 
Besides, for the stopping criterion, we simply set a maximum number of adaptation steps: 7 for 2D problems and 4 for 3D problems.

We use the $L^2(\Omega)$ and $H^1(\Omega)$ error norms to evaluate the accuracy of the identified solution $u^{id}$:
\begin{equation*}
    e_{L^2(\Omega)}(u^{id}) = \ltwonorm{u - u^{id}}, \qquad 
    e_{H^1(\Omega)}(u^{id}) = \honenorm{u - u^{id}},
\end{equation*}
where $u$ is the true solution, $\ltwonorm{\cdot} = \sqrt{\int_\Omega |\cdot|^2}$, and $\honenorm{\cdot} = \sqrt{\int_\Omega |\cdot|^2 + |\pmb{\nabla}(\cdot)|^2}$. 
The integrals in these terms are evaluated with sufficient Gauss quadrature points to guarantee accuracy. 
It should be noted that $u^{id}$ can either be a \ac{nn} (i.e., $u_{\mathcal{N}}$) or its interpolation (i.e., $\tilde{\pi}_h(u_{\mathcal{N}}) + \bar{u}_h$) onto a suitable \ac{fe} space. In the different plots in this section, we label the \ac{nn} interpolation as ``FEINN'' and refer to the \ac{nn} itself with an additional tag ``\ac{nn} only''. From now on, we drop $\Omega$ in the notation for the norms for conciseness.

The unpreconditioned loss~\eqref{eq:feinn_discrete_loss} has proven effective in various cases, as shown in~\cite{Badia2024}. Therefore, we use it as the default loss function in the experiments. 
The preconditioned loss function~\eqref{eq:feinn_precond_loss} is only used in \sect{subsubsec:arc_preconditioner_study} to examine the effectiveness of the proposed preconditioning strategy.
With regards to the norm choice in the loss~\eqref{eq:feinn_discrete_loss}, we use the $\ell^1$ norm. We note that other norms can also be explored. For example, our previous work~\cite{Badia2024} shows that both $\ell^2$ and $\ell^1$ norms in~\eqref{eq:feinn_discrete_loss} produce similar results on (quasi-)uniform meshes. 
Besides, we explore different norm options in the preconditioned loss~\eqref{eq:feinn_precond_loss}, as detailed in \sect{subsubsec:arc_preconditioner_study}.

The main motivation of this work is to showcase the capability of the proposed method in handling PDEs with sharp gradients or singularities. For this purpose, we have considered standard tests on simple geometries. However, as demonstrated in our previous work~\cite{Badia2024}, the \ac{feinn} method can also be applied to problems posed on more complex domains, and $h$-adaptive \acp{feinn} inherit this capability seamlessly. 
To attack this kind of multi-scale problems, one can generate a forest-of-octrees built out 
of a coarse mesh $\mathcal{C}$ generated by an unstructured mesh generator such as, e.g., \texttt{Gmsh}~\cite{Gmsh2009}; see \sect{subsec:auto_mesh_adapt}.

\subsection{The 2D arc wavefront problem with sharp gradients}  \label{subsec:arc_wavefront}
\begin{figure}[h]
    \begin{subfigure}[t]{0.32\linewidth}
        \centering
        \includegraphics[height=0.88\textwidth]{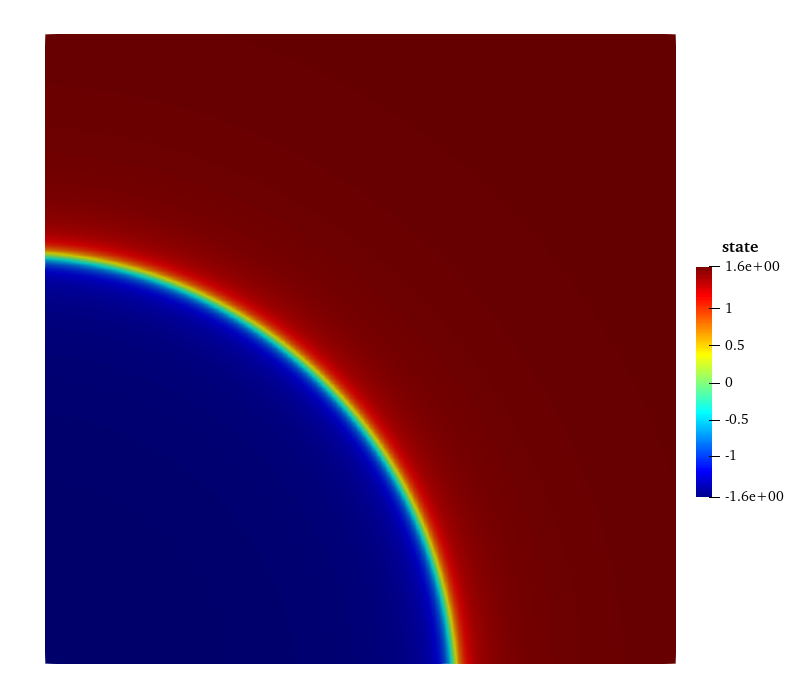}
        \caption{$u$}
        \label{fig:arc_true_state}
    \end{subfigure}
    \begin{subfigure}[t]{0.32\linewidth}
        \centering
        \includegraphics[height=0.88\textwidth]{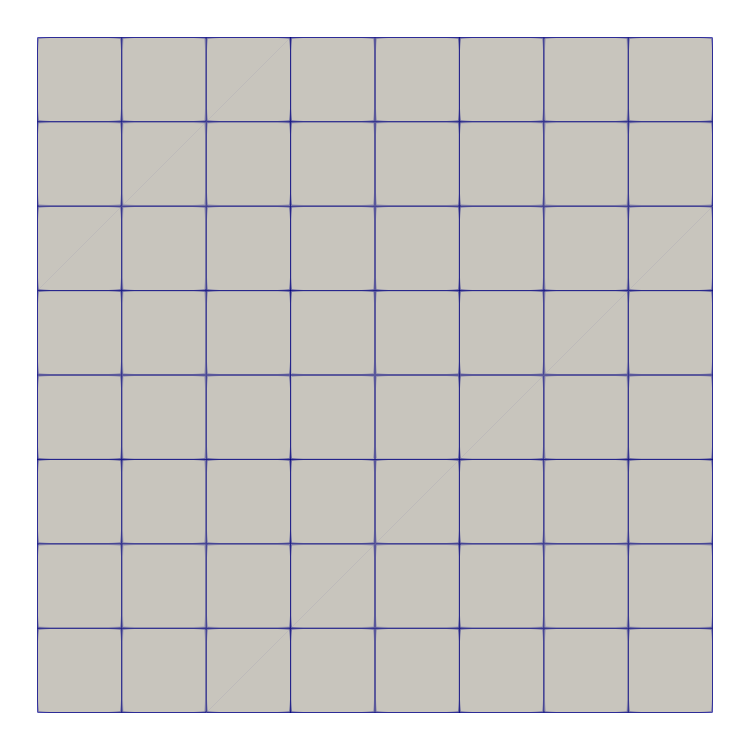}
        \caption{Initial mesh}
        \label{fig:arc_order4_init_mesh}
    \end{subfigure}
    \begin{subfigure}[t]{0.32\linewidth}
        \centering
        \includegraphics[height=0.88\textwidth]{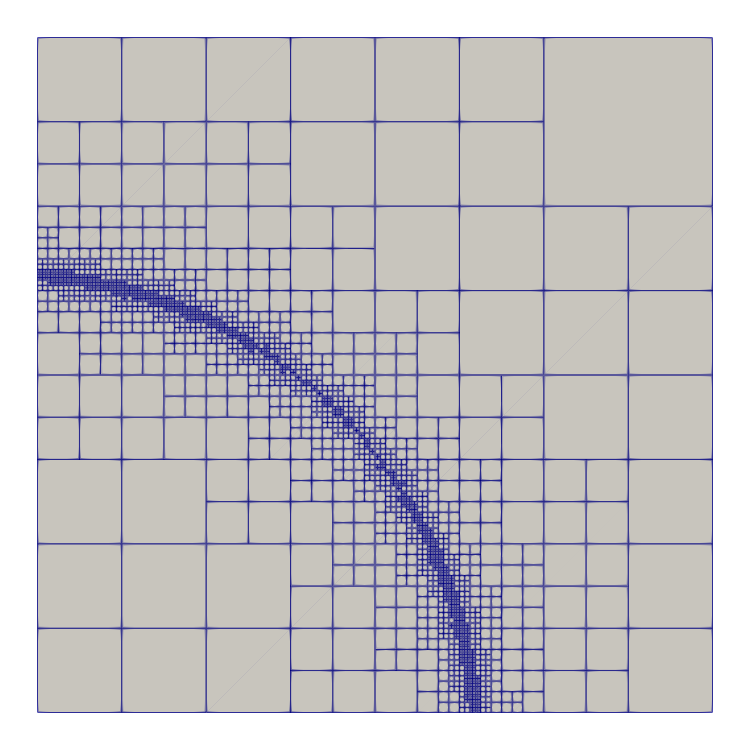}
        \caption{$h$-Adaptive \ac{fem} final mesh}
        \label{fig:arc_order4_fem_final_mesh}
    \end{subfigure}
     
    \caption{The true solution, initial mesh, and $h$-adaptive \ac{fem} final mesh for the 2D arc wavefront problem. The final mesh is obtained through $h$-adaptive \ac{fem} with 7 iterative adaptation steps using the real \ac{fem} error as the error indicator.}
    \label{fig:arc_state_and_meshes}
\end{figure}

Let us first consider a smooth problem with sharp gradients.  
Specifically, the problem is defined on a square domain $\Omega = [0, 1]^2$. We pick suitable $f$ and $g$ such that the exact solution is: 
\begin{equation*}
    u(x,y)= \arctan(100(\sqrt{(x+0.05)^2+(y+0.05)^2}-0.7)).
\end{equation*}
As shown in \fig{fig:arc_state_and_meshes}a, the true solution steeply varies in the neighbourhood of the arc-shaped front, resulting in sharp gradients along the arc. For the domain discretisation, we use a forest-of-quadtrees built out of a coarse mesh $\mathcal{C}$ with a single quadrilateral/adaptive quadtree.

We initiate the training of the \acp{nn} with $k_U=4$ on the initial mesh depicted in \fig{fig:arc_state_and_meshes}b. 
This mesh results from the application of three levels of uniform mesh refinement to $\mathcal{C}$.
For illustration purposes, we solve the problem using standard $h$-adaptive \ac{fem} with 7 mesh adaptation cycles, and the resulting mesh is shown in \fig{fig:arc_state_and_meshes}c. The error indicator for the $h$-adaptive \ac{fem} used to generate this mesh is the real \ac{fem} error.

\subsubsection{The comparison of error indicators} \label{subsubsec:arc_indicator_study}

In the first experiment, we investigate the impact of the choice of error indicator on the results. We repeat the experiment 20 times, each time with a different \ac{nn} initialisation, for every error indicator. The errors versus adaptation steps for both \acp{nn} and their interpolations are illustrated in \fig{fig:arc_indicator_study_error_convergence}. The solid line represents the median, and the band illustrates the range from the 0th to the 90th percentiles across 20 runs. The red dashed line corresponds to the $h$-adaptive \ac{fem} using the real \ac{fem} error for mesh adaptation.

\begin{figure}[h]
    \centering
    \includegraphics[width=\textwidth]{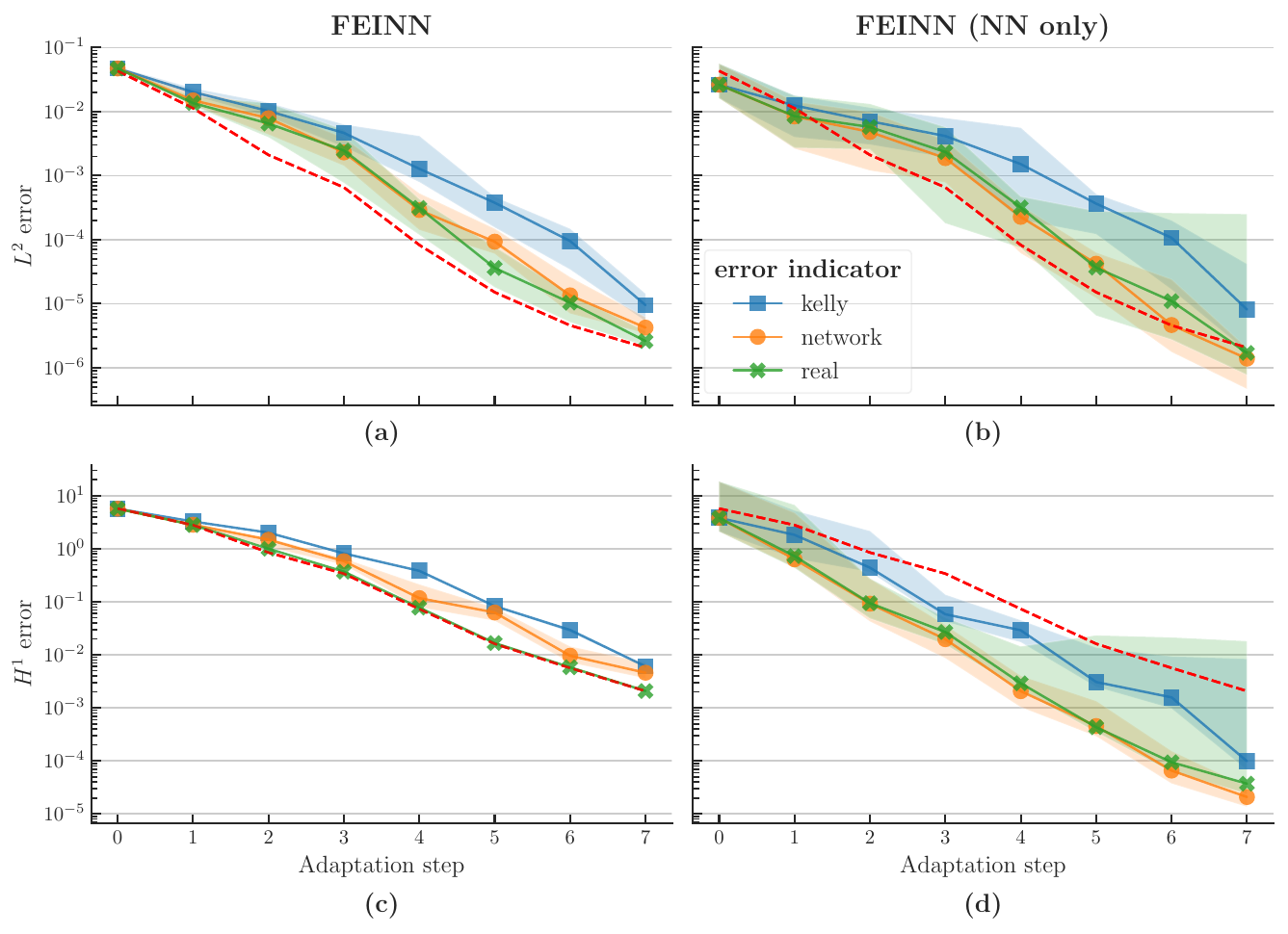}
    \caption{Convergence of $u^{id}$ in $L^2$ and $H^1$ errors for the 2D arc wavefront problem, using different error indicators. The top row shows $L^2$ errors and the bottom row represents $H^1$ errors. The first column displays errors of the interpolated \acp{nn} and the second column corresponds to errors of the \acp{nn} themselves. The line denotes the median, and the band represents the range from the minimum to the 90th percentile across 20 independent runs. The red dashed line illustrates the error of the $h$-adaptive \ac{fem} solution, using the real \ac{fem} error as the error indicator.}
    \label{fig:arc_indicator_study_error_convergence}
\end{figure}

Let us first comment on the results for the interpolated \acp{nn} in the first column of \fig{fig:arc_indicator_study_error_convergence}. When using the real error as the mesh adaptation criterion, both $L^2$ and $H^1$ errors are consistently the smallest at each adaptation step. Notably, for $H^1$ errors, the \ac{feinn} error line aligns with the one of $h$-adaptive \ac{fem}. 
The other two error indicators are also very effective for the interpolated \acp{nn}, as indicated by their stably decreasing $L^2$ and $H^1$ error lines.
Besides, the bands of the \ac{feinn} errors, regardless of the error indicator used, are very narrow, underscoring the robustness of \acp{feinn} with respect to different error indicators and \ac{nn} initialisations.

Regarding the errors in \acp{nn} themselves, as shown in the second column of \fig{fig:arc_indicator_study_error_convergence}, we observe that the majority of the $L^2$ errors from the 20 experiments are within an acceptable range. This is evident as all the lines (medians) corresponding to different error indicators closely align with those of the interpolated \acp{nn}. Nevertheless, the bands for $L^2$ and $H^1$ errors become noticeably wider, suggesting larger variance in the results. The network error indicator stands out for its robustness, as evidenced by the narrowest bands in \fig{fig:arc_indicator_study_error_convergence}b and ~\ref{fig:arc_indicator_study_error_convergence}d. In our previous work~\cite{Badia2024}, we have highlighted that for a smooth true solution, \acp{nn} have the potential to outperform the \ac{fem} solution on the same mesh. We observe a similar phenomenon here concerning $H^1$ errors: with real error indicator, \acp{nn} without interpolation can beat \ac{fem} solution (represented by the red dashed line) by approximately 2 orders of magnitude at the last adaptation step. For the remaining two error indicators, the majority of the $H^1$ errors fall below the \ac{feinn} line, indicating improved performance of the \acp{nn} compared to their interpolation.

In \fig{fig:arc_feinn_final_meshes}, we show the final meshes corresponding to the three different error indicators used in this study. 
We represent, for instance, the mesh adapted using the Kelly error indicator at adaptation step 7 as $\mathcal{T}_7^k$. In this notation, the letter in the superscript indicates the error indicator (``k'' for Kelly, ``n'' for ``network'', and ``r'' for ``real''), while the number denotes the adaptation step identifier. 
Notably, all error indicators effectively identify the region with sharp gradients and appropriately refine the mesh in that specific area. 
Nevertheless, differences also exist in these meshes.
Firstly, $\mathcal{T}_7^k$ (\fig{fig:arc_feinn_final_meshes}a) with 1,336 elements and 20,297 \acp{dof} has the fewest elements, accounting for the largest errors identified for Kelly in \fig{fig:arc_indicator_study_error_convergence} compared to other error indicators. This highlights that, in this problem, the network error indicator outperforms the Kelly error indicator.
Subsequently, $\mathcal{T}_7^r$ (\fig{fig:arc_feinn_final_meshes}c) with 2,146 elements and 32,893 \acp{dof} closely resembles the $h$-adaptive \ac{fem} final mesh with 2,119 elements and 32,457 \acp{dof} in \fig{fig:arc_feinn_final_meshes}c, elucidating the coincidence of the real indicator $H^1$ error line and the \ac{fem} line in \fig{fig:arc_indicator_study_error_convergence}c. 
Lastly, although $\mathcal{T}_7^n$ (\fig{fig:arc_feinn_final_meshes}b) with 2,293 elements and 35,089 \acp{dof} is as dense as $\mathcal{T}_7^r$ (\fig{fig:arc_feinn_final_meshes}c), the patterns are slightly different: $\mathcal{T}_7^n$ contains more elements in the middle of the arc but fewer elements at two ends of the arc, while the opposite holds true for $\mathcal{T}_7^r$. This discrepancy could explain the  differences observed in the convergence curves (\fig{fig:arc_indicator_study_error_convergence}) for real and network error indicators.

\begin{figure}[h]
    \begin{subfigure}[t]{0.32\linewidth}
        \centering
        \includegraphics[height=0.88\textwidth]{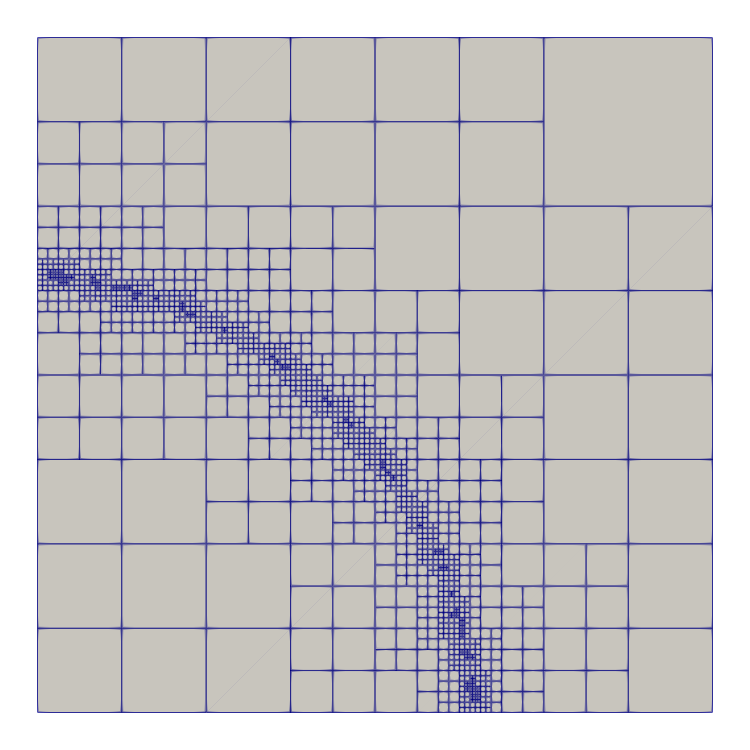}
        \caption{$\mathcal{T}_7^k$}
        \label{fig:arc_kelly_indicator_final_mesh}
    \end{subfigure}
    \begin{subfigure}[t]{0.32\linewidth}
        \centering
        \includegraphics[height=0.88\textwidth]{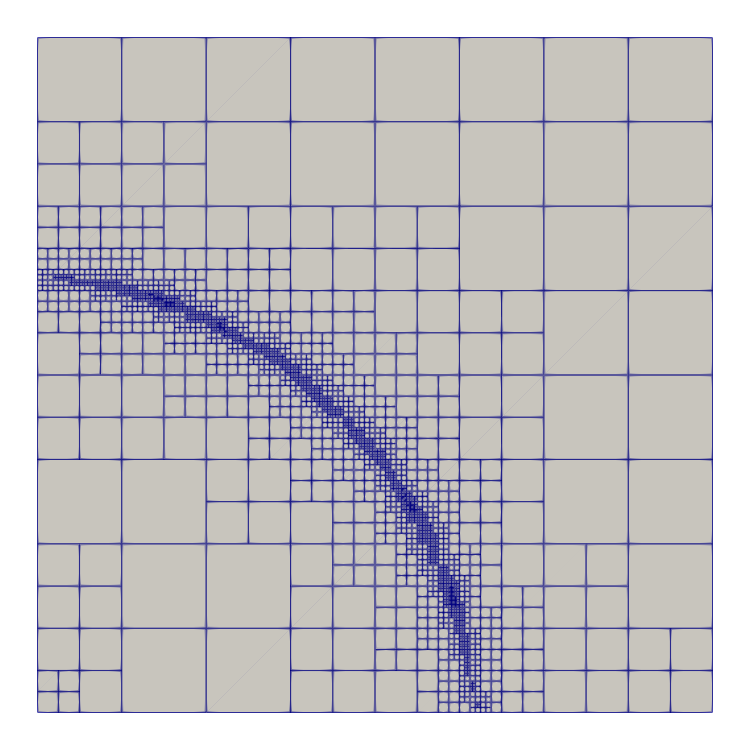}
        \caption{$\mathcal{T}_7^n$}
        \label{fig:arc_network_indicator_final_mesh}
    \end{subfigure}
    \begin{subfigure}[t]{0.32\linewidth}
        \centering
        \includegraphics[height=0.88\textwidth]{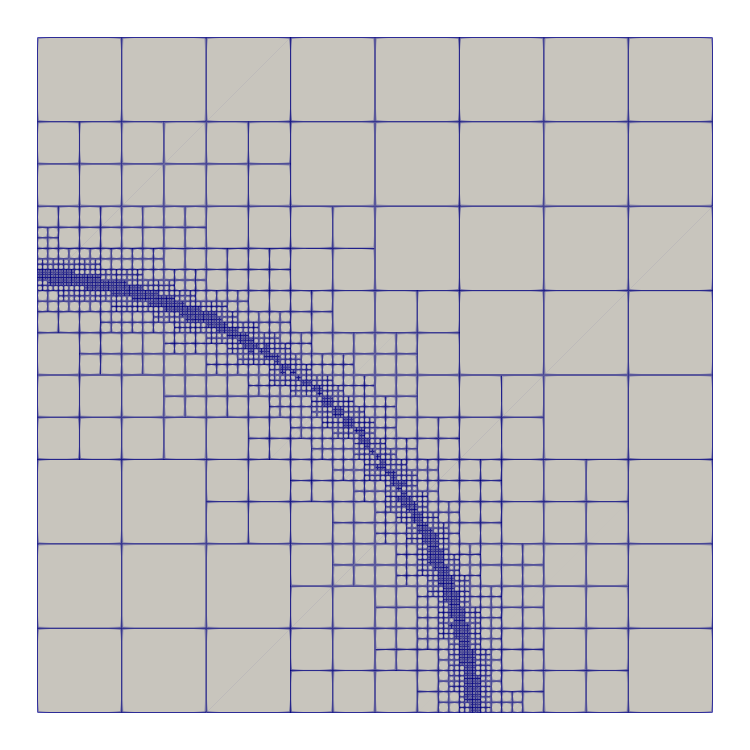}
        \caption{$\mathcal{T}_7^r$}
        \label{fig:arc_real_indicator_final_mesh}
    \end{subfigure}
     
    \caption{Comparison of final meshes obtained by training of the $h$-adaptive \acp{feinn} for the 2D arc wavefront problem. The meshes result from 7 mesh adaptation steps using Kelly, network, and real error indicators, respectively.}
    \label{fig:arc_feinn_final_meshes}
\end{figure}

\subsubsection{The effect of preconditioning} \label{subsubsec:arc_preconditioner_study}
In the  experiments so far, we have observed a notable  variance with respect to parameter initialisation in the errors corresponding to non-interpolated \acp{nn}, particularly when using Kelly and real error indicators. 
In this experiment, we investigate whether preconditioning can effectively enhance robustness (i.e., reduce this variance) and accelerate the training process. We concentrate on the impact of the preconditioner on the results using the Kelly error indicator, as accessing to the true error is often impossible in practice. Besides, we also note that 
the results obtained with the network error indicator are very similar to those obtained with the Kelly error indicator, thus we omit from this section the detailed results obtained with the former indicator for brevity.

We consider the preconditioner $\mathbf{B}_{\rm inv\_lin}$, which is defined as the discrete Laplacian resulting from discretisation with the linearised space $V_h$ used both as trial and test FE spaces. Thus, $\mathbf{B}_{\rm inv\_lin}$ is a sparse \ac{spd} matrix with reduced computational cost compared to a regular sparse matrix. Note that we apply $k_U$ levels of uniform refinement for the partition associated to $U_h$ to obtain the mesh for the \emph{linearised} test space $V_h$. 
Besides, in~\cite{Badia2024}, we have observed the effectiveness of a \ac{gmg} preconditioner as an approximation to the inverse of $\mathbf{B}_{\rm inv\_lin}$, suggesting its potential benefits for this problem as well. Nonetheless, for the purpose of this study, we just exactly invert $\mathbf{B}_{\rm inv\_lin}$ (by solving linear systems), as it is sufficient to demonstrate the impact of preconditioning.
In these experiments, we cut the number of training iterations by reducing those used in the unpreconditioned experiments by a factor of 4, except for the final adaptation step, where these are divided by a factor of 2. This adjustment accommodates for the increased demand for training on the finest meshes.

In addition to evaluating to what extent preconditioners can expedite convergence, we also investigate the impact of norm choice on the convergence. 
We train the same \ac{nn} (i.e., the same parameter initialisation) using the preconditioned loss~\eqref{eq:feinn_precond_loss} equipped with 4 different norms: the $L^1$ and $L^2$ norm of the discrete Riesz map of the residual and its gradient; we denote the $L^1$ and $L^2$ norms on the gradient as $W^{1,1}$ and $W^{1,2}$ norms, since $W^{k,p}$ represents the Sobolev norm with $k$ derivatives and $p$-th power. 
In \fig{fig:arc_precond_study_kelly_indicator_error_history}, we report the $L^2$ and $H^1$ errors as a function of  training iteration for the four norms. For reference, we also include the results for the unpreconditioned case, labelled as ``N/A''.
The first thing to note is that all the preconditioned curves are much steeper than the unpreconditioned ones, indicating that the application of the preconditioner significantly accelerates the convergence of both \acp{nn} and their interpolations.
Moreover, the $L^2$ convergence for both the \ac{nn} and its interpolation is highly unstable when a preconditioner is not applied, as evidenced by the error curves' fluctuating pattern. 
In contrast, the preconditioned error curves are all steadily decreasing, suggesting that the preconditioner effectively stabilises the training process.
This also reinforces the conclusion from our previous work~\cite{Badia2024} that using the $L^2$ norm in the preconditioned loss can accelerate the convergence of training in terms of $L^2$ error. 
Furthermore, the $H^1$ convergence of the interpolated \acp{nn} rapidly decreases in both unpreconditioned and preconditioned cases, as indicated by the stair-like pattern of the convergence curves.
However, upon comparing \fig{fig:arc_precond_study_kelly_indicator_error_history}c and \fig{fig:arc_precond_study_kelly_indicator_error_history}d, it becomes apparent that the convergence of \ac{nn} and its interpolation are not at the same pace. Additional iterations are necessary to ensure the convergence of the \ac{nn}, even if its interpolation has already reached a favourable position.
To sum up, when using the preconditioner $\mathbf{B}_{\rm inv\_lin}$, all the four norms in~\eqref{eq:feinn_precond_loss} are effective, with the $W^{1,1}$ norm being slightly better in terms of convergence speed and stability.

\begin{figure}[h]
    \centering
    \includegraphics[width=\textwidth]{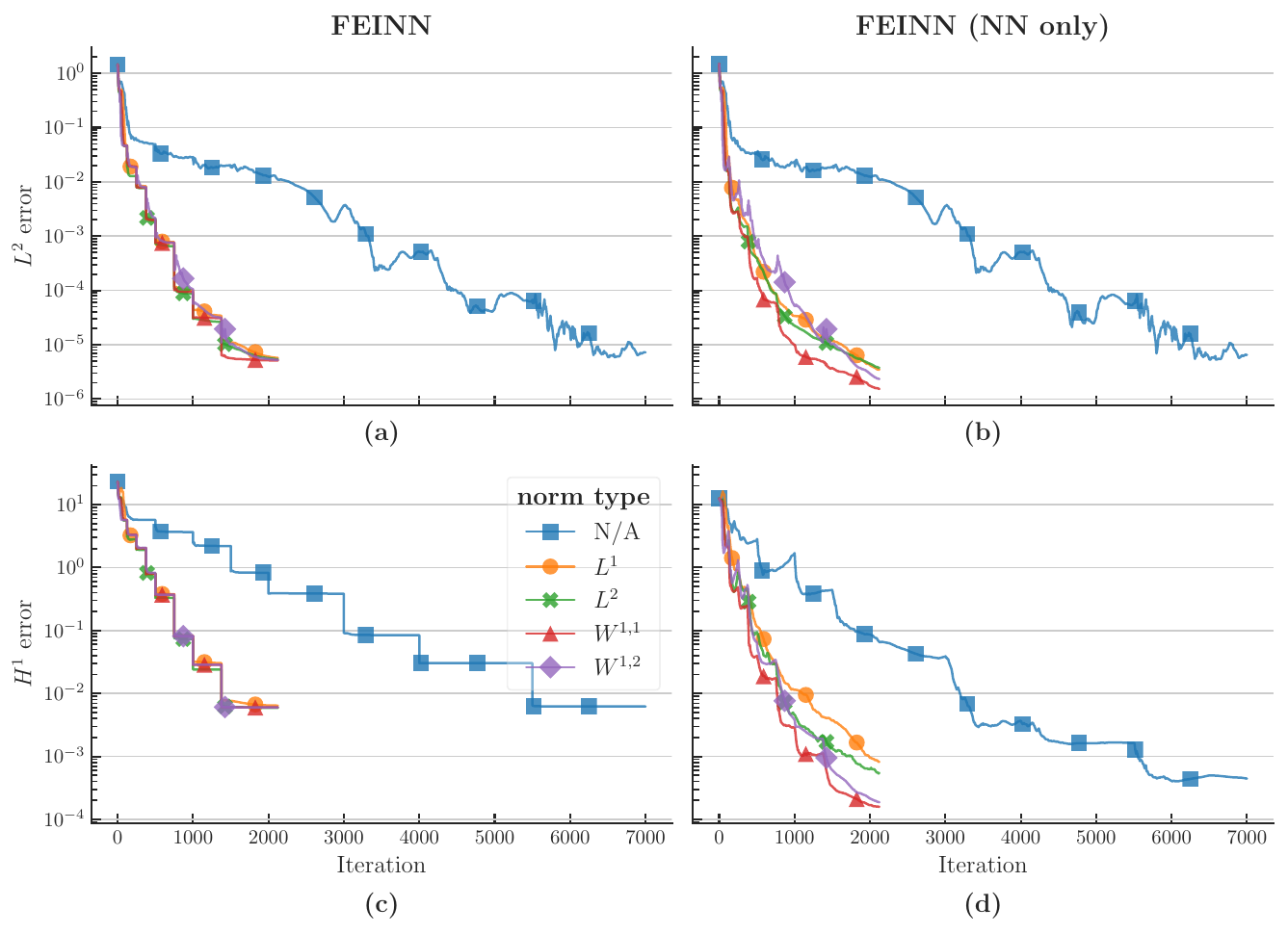}
    \caption{Convergence history of $u^{id}$ in $L^2$ and $H^1$ errors for the 2D arc wavefront problem, using different norms in the preconditioned loss during training. Kelly error indicator is used for mesh adaptation. The top row shows $L^2$ errors and the bottom row represents $H^1$ errors. The first column displays errors of the interpolated \acp{nn} and the second column corresponds to errors of the \acp{nn} themselves.}
    \label{fig:arc_precond_study_kelly_indicator_error_history}
\end{figure}

\begin{figure}[h]
    \centering
    \includegraphics[width=\textwidth]{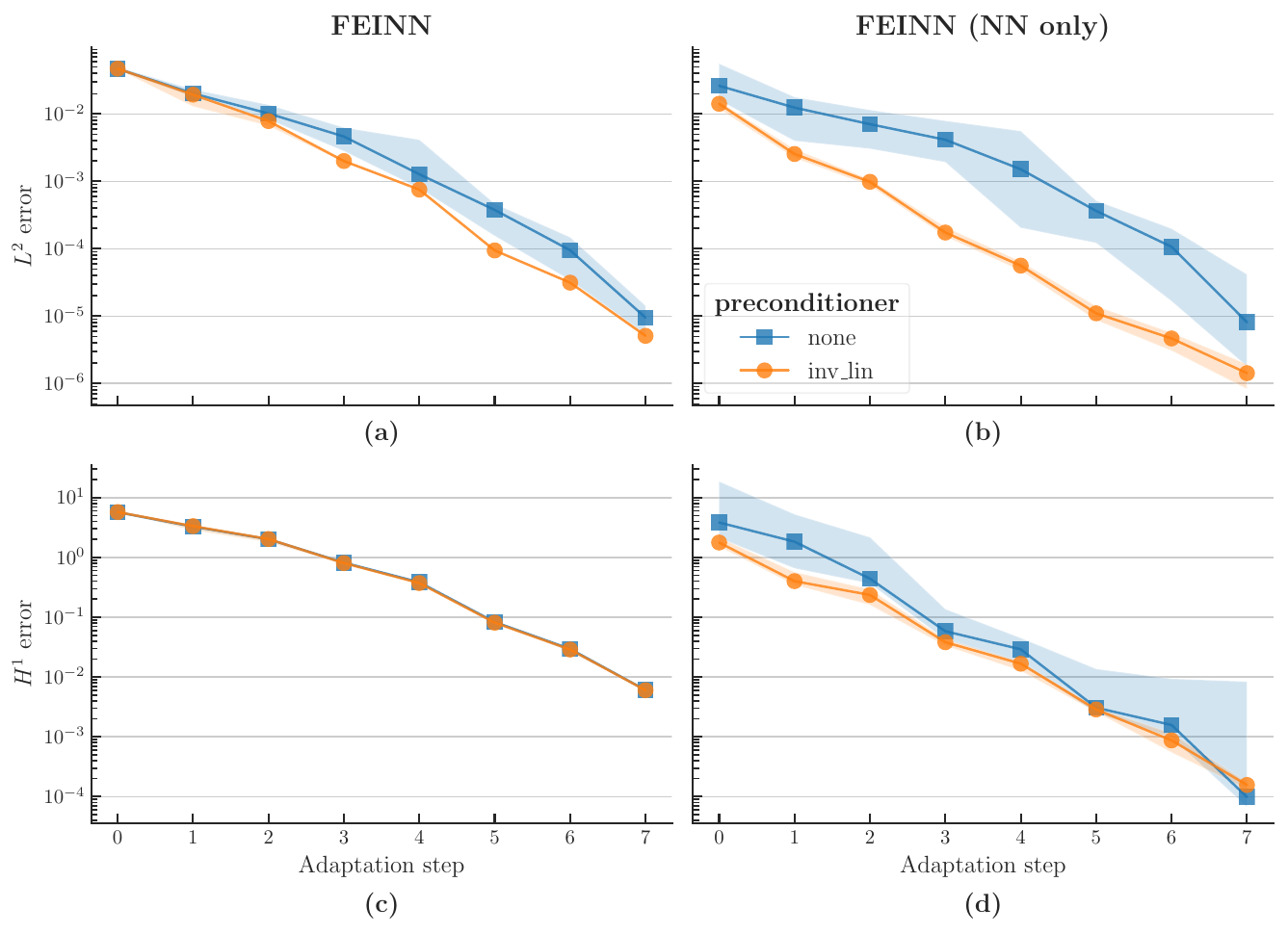}
    \caption{Convergence of $u^{id}$ in $L^2$ and $H^1$ errors for the 2D arc wavefront problem, using different preconditioners during training. Kelly error indicator is used for mesh adaptation. Refer to the caption of \fig{fig:arc_indicator_study_error_convergence} for details on the information being displayed in this figure.}
    \label{fig:arc_precond_study_kelly_indicator_error_convergence}
\end{figure}

We now study how preconditioning impacts training robustness with respect to \ac{nn} initialisation. We use the same 20 \ac{nn} initial parameters as in \sect{subsubsec:arc_indicator_study} and rerun the experiment using the preconditioned loss~\eqref{eq:feinn_precond_loss} equipped with the preconditioner $\mathbf{B}_{\rm inv\_lin}$ and the $W^{1,1}$ norm.
The error convergence using the preconditioned and unpreconditioned loss functions versus mesh adaptation step is depicted in \fig{fig:arc_precond_study_kelly_indicator_error_convergence}. The legend label ``none'' means no preconditioning, while ``inv\_lin'' represents that the preconditioner at hand is applied.
We can extract several key findings from the figure. Firstly, the application of $\mathbf{B}_{\rm inv\_lin}$ notably enhances the performance of both the \acp{nn} and their \ac{fe} interpolations. This is evident from the consistently lower position of the preconditioned lines compared to the unpreconditioned lines in \fig{fig:arc_precond_study_kelly_indicator_error_convergence}. Although the preconditioned \ac{feinn} $H^1$ results may not surpass their unpreconditioned counterparts, they exhibit close similarity, suggesting convergence to the same solution in both cases.
Secondly, the preconditioner $\mathbf{B}_{\rm inv\_lin}$ significantly improves the training robustness, evident in thinner bands for the preconditioned results, especially notable for the \acp{nn} without interpolation (second column in \fig{fig:arc_precond_study_kelly_indicator_error_convergence}). 
Lastly, $\mathbf{B}_{\rm inv\_lin}$ also helps \acp{nn} surpass their interpolations in terms of $L^2$ results, suggested by the lower position of the preconditioned line for \ac{nn} $L^2$ errors in \fig{fig:arc_precond_study_kelly_indicator_error_convergence}b compared to the \ac{feinn} $L^2$ errors in \fig{fig:arc_precond_study_kelly_indicator_error_convergence}a.

In summary, the proposed preconditioning strategy proves very effective for $h$-adaptive \acp{feinn}. It expedites the training process and enhances the robustness of the results with respect to \ac{nn} initialisations.

\subsection{The 2D Fichera problem with singularity} \label{subsec:lshape}
In this subsection, we consider the Fichera problem defined on the L-shaped domain $\Omega = [-1,1]^2 \backslash [-1,0]^2$. For the discretisation, we use a forest-of-quadtrees built out of a coarse mesh $\mathcal{C}$ with three quadrilateral/adaptive quadtrees. We pick $f$ and $g$ such that the true solution is:
\begin{equation*}
    u(r, \theta) = r^{\frac{2}{3}}\sin(\frac{2}{3}(\theta + \frac{\pi}{2})),
\end{equation*}
expressed in polar coordinates. 
The true solution and the domain are displayed in \fig{fig:lshape_state_and_meshes}a. Note that derivatives of $u$ are singular at the origin, in particular, $u \in H^{5/3 - \epsilon}(\Omega)$ for any $\epsilon > 0$.
We have tackled the same problem on the unit square in~\cite{Badia2024}, noting that while the interpolated \acp{nn} can recover the \ac{fem} solution on a fixed mesh, there is still room for their generalisation to be improved. Therefore, our goal in this experiment is to assess whether the proposed adaptive training strategy improves the \acp{nn}' generalisation.

\begin{figure}[h]
    \begin{subfigure}[t]{0.32\linewidth}
        \centering
        \includegraphics[height=0.88\textwidth]{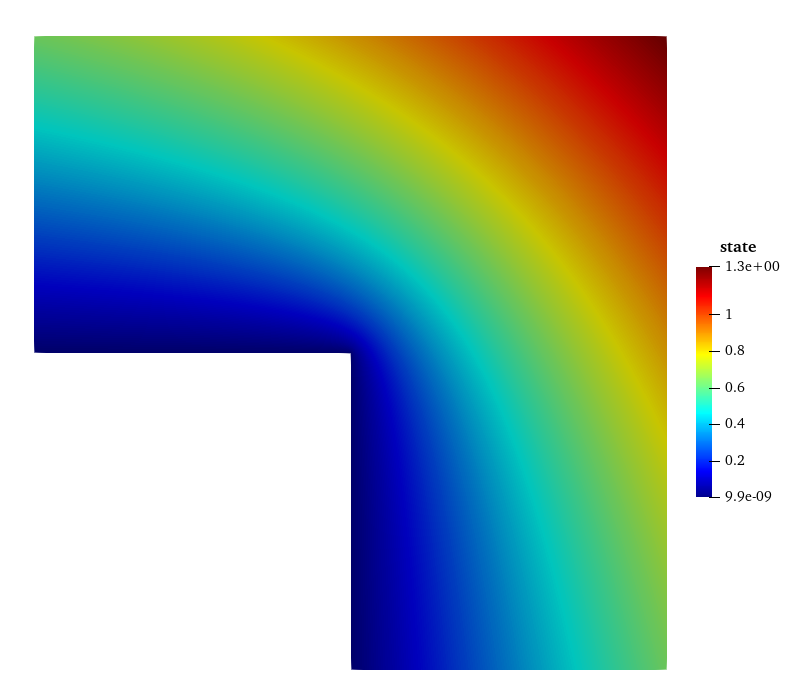}
        \caption{$u$}
        \label{fig:lshape_true_state}
    \end{subfigure}
    \begin{subfigure}[t]{0.32\linewidth}
        \centering
        \includegraphics[height=0.88\textwidth]{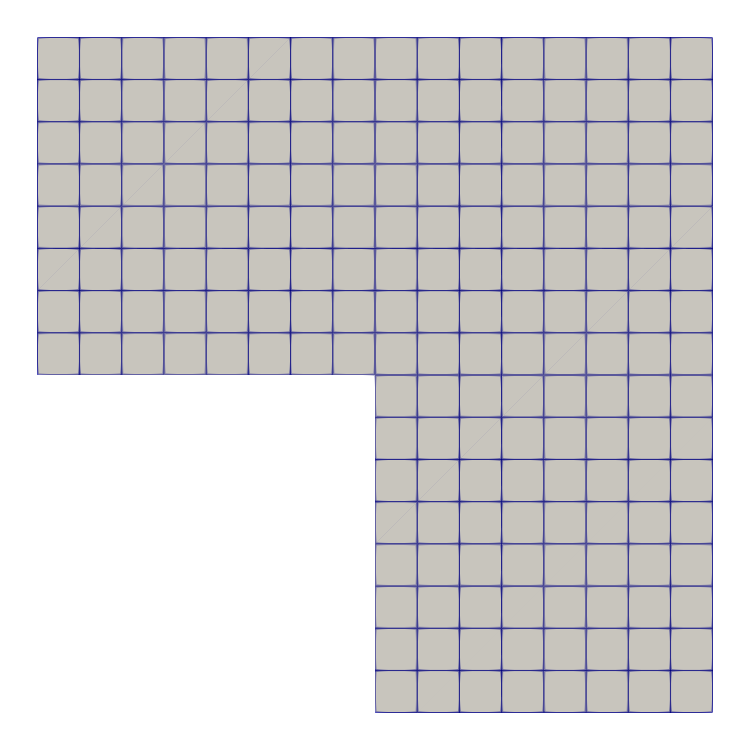}
        \caption{Initial mesh}
        \label{fig:lshape_order4_init_mesh}
    \end{subfigure}
    \begin{subfigure}[t]{0.32\linewidth}
        \centering
        \includegraphics[height=0.88\textwidth]{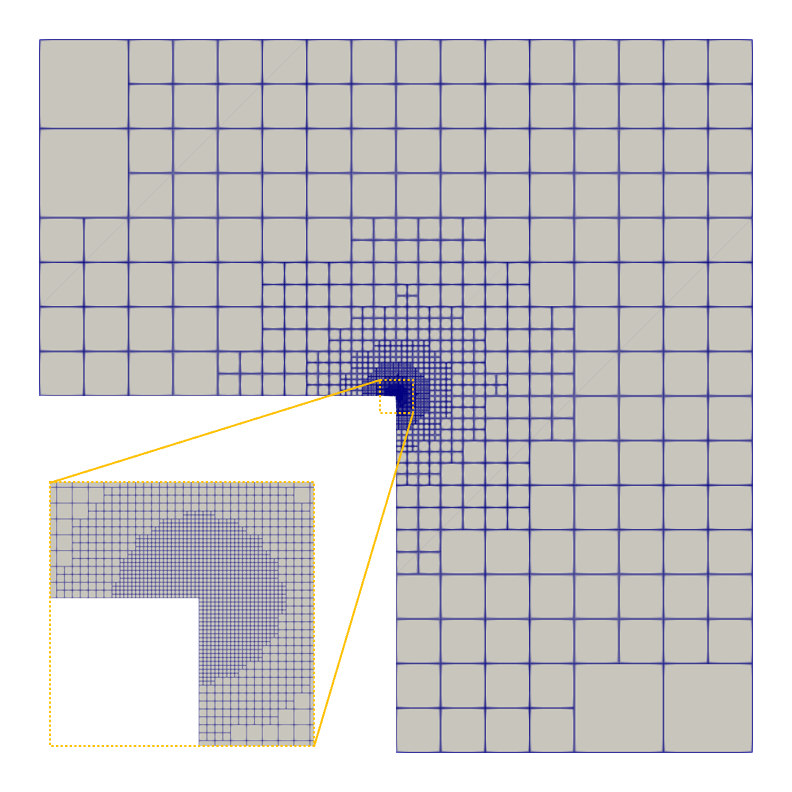}
        \caption{$h$-Adaptive \ac{fem} final mesh}
        \label{fig:lshape_order4_fem_final_mesh}
    \end{subfigure}
     
    \caption{The true solution, initial mesh, and $h$-adaptive \ac{fem} final mesh for the 2D problem with singularity defined on the L-shaped domain. The final mesh is obtained through the $h$-adaptive \ac{fem} with 7 mesh adaptation steps using the real \ac{fem} error as the error indicator for illustration purposes.}
    \label{fig:lshape_state_and_meshes}
\end{figure}

\subsubsection{The comparison of error indicators} \label{subsubsec:lshaped_indicator_study}

Let us first examine the effectiveness of all three error indicators. Starting with an initial mesh shown in \fig{fig:lshape_state_and_meshes}b (resulting from the application of 4 levels of uniform refinement to the initial coarse mesh $\mathcal{C}$), we perform 20 independent runs for each indicator, and present the errors versus the adaptation step in \fig{fig:lshape_indicator_study_error_convergence}. We choose $k_U=4$ in these experiments.  
The first observation from the figure is that the performance of all three error indicators is comparable, as the lines representing different indicators closely overlap.
Besides, for interpolated \acp{nn}, the distinction among different error indicators is visually indistinguishable. The $L^2$ error curves (\fig{fig:lshape_indicator_study_error_convergence}a) and $H^1$ error curves (\fig{fig:lshape_indicator_study_error_convergence}c) for all indicators closely follow the red dashed $h$-adaptive \ac{fem} line. The variance in errors at each adaptation step is also very low, as indicated by very narrow bands.
Additionally, the \acp{nn} exhibit strong generalisation in terms of $L^2$ error, with all lines in \fig{fig:lshape_indicator_study_error_convergence}c closely aligning with the \ac{fem} line and featuring very narrow bands.
More importantly, the $h$-adaptive \ac{feinn} method enables \acp{nn} to gradually learn the singularity and achieve the $h$-adaptive \ac{fem} rate of convergence, as shown in \fig{fig:lshape_indicator_study_error_convergence}d. Furthermore, the variance in the $H^1$ errors for \acp{nn} is minimal, as suggested by the narrow bands.
In brief, compared to the \ac{feinn} method defined on a fixed mesh in our previous work~\cite{Badia2024}, the $h$-adaptive \acp{feinn} method improves the performance of the \acp{nn} themselves and enable them to effectively capture the singularity, irrespective of the error indicator employed.

\begin{figure}[h]
    \centering
    \includegraphics[width=\textwidth]{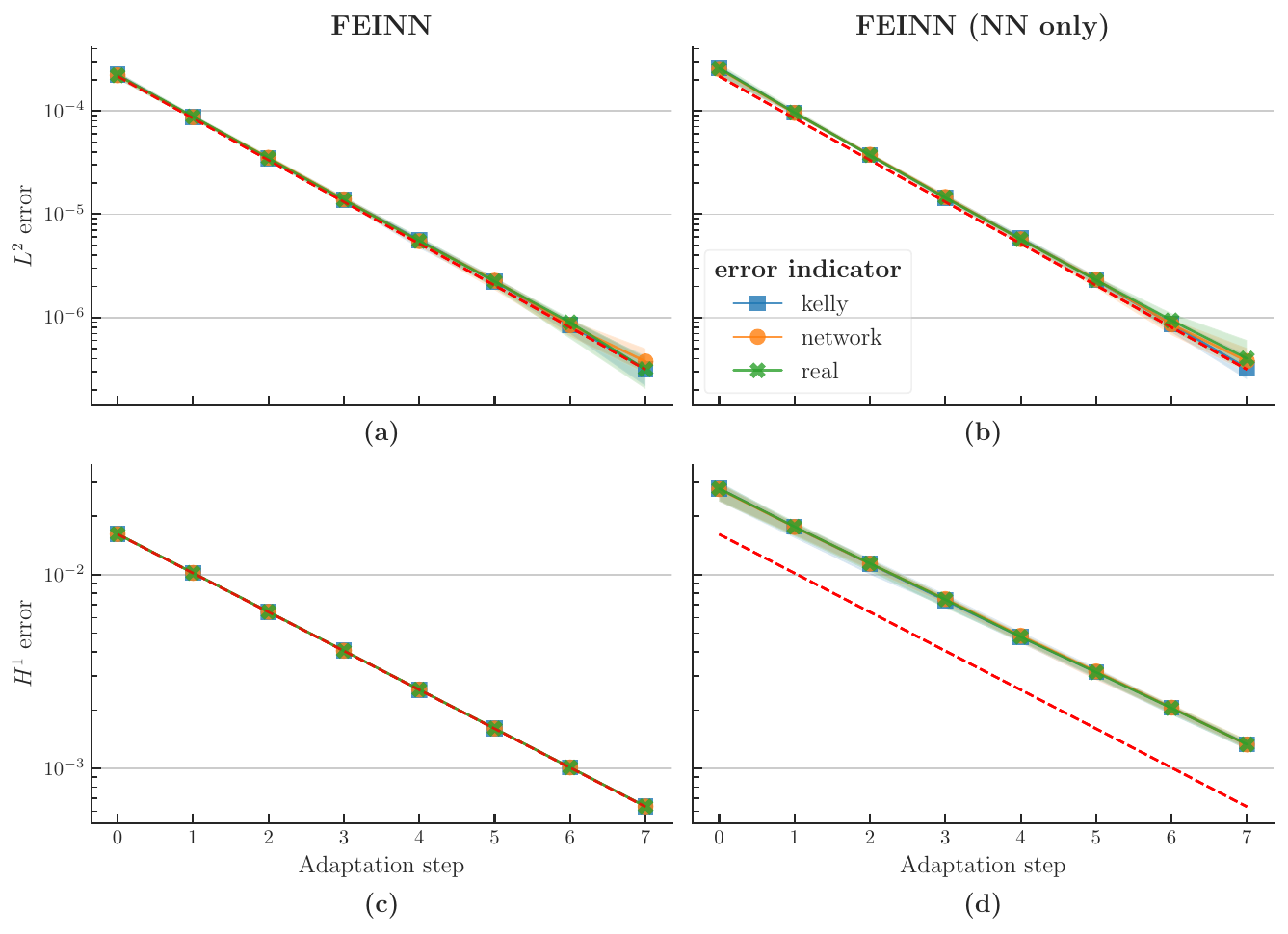}
    \caption{Convergence of $u^{id}$ in $L^2$ and $H^1$ errors for the 2D Fichera problem with singularity , using different error indicators. See caption of \fig{fig:arc_indicator_study_error_convergence} for details on the information being displayed in this figure.}
    \label{fig:lshape_indicator_study_error_convergence}
\end{figure}

In \fig{fig:lshape_feinn_final_meshes}, we show the final meshes obtained with different error indicators. 
Once again, we notice the striking similarity between $\mathcal{T}_7^r$ (\fig{fig:lshape_feinn_final_meshes}c) with 2,610 elements and 40,841 \acp{dof} and the $h$-adaptive \ac{fem} final mesh (\fig{fig:lshape_state_and_meshes}c) with 2,637 elements and 41,313 \acp{dof}.
Although the mesh patterns of $\mathcal{T}_7^k$ (\fig{fig:lshape_feinn_final_meshes}a) with 2,637 elements and 40,705 \acp{dof} and $\mathcal{T}_7^n$ (\fig{fig:lshape_feinn_final_meshes}b) with 2,703 elements and 41,781 \acp{dof} are different from $\mathcal{T}_7^r$, the elements around the singular point are refined to the same level across all meshes. This is clearer when comparing the zoomed view of the corner for the different meshes in these figures. Besides, the differences in the number of elements and \acp{dof} are negligible, suggesting that training cost is similar for all error indicators.
In the context of this problem and experiment, with high $k_U$ and finer initial mesh, the primary contribution to the $H^1$ errors arises from the singularity. 
Consequently, despite the differences in mesh patterns, the same element size around the singular point leads to very similar $H^1$ errors. This observation is reinforced by the remarkably similar $H^1$ convergence curves for different error indicators shown in \fig{fig:lshape_indicator_study_error_convergence}d.

\begin{figure}[h]
    \begin{subfigure}[t]{0.32\linewidth}
        \centering
        \includegraphics[height=0.88\textwidth]{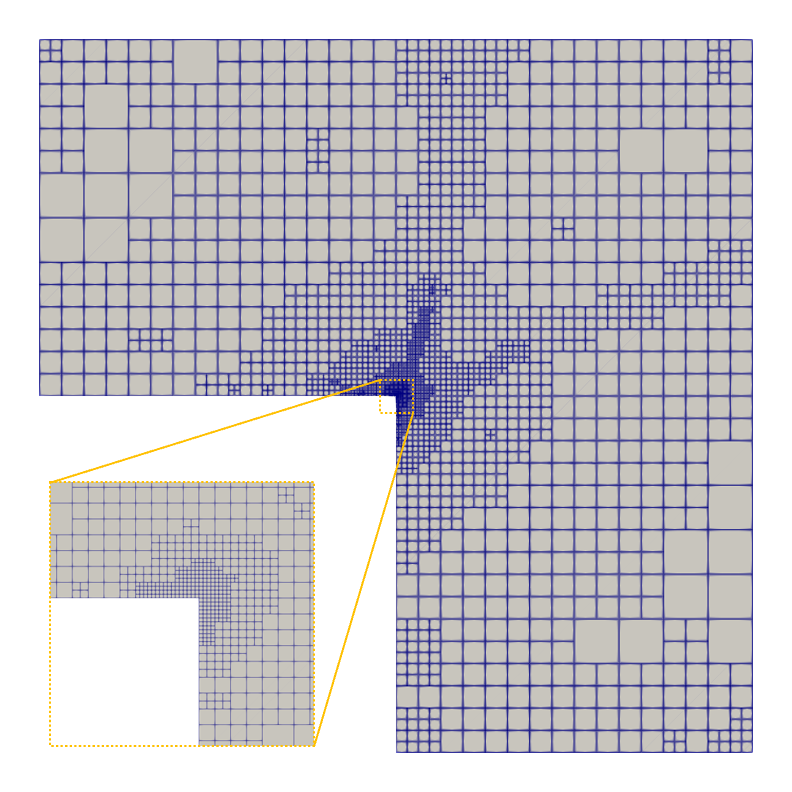}
        \caption{$\mathcal{T}_7^k$}
        \label{fig:lshape_kelly_indicator_final_mesh}
    \end{subfigure}
    \begin{subfigure}[t]{0.32\linewidth}
        \centering
        \includegraphics[height=0.88\textwidth]{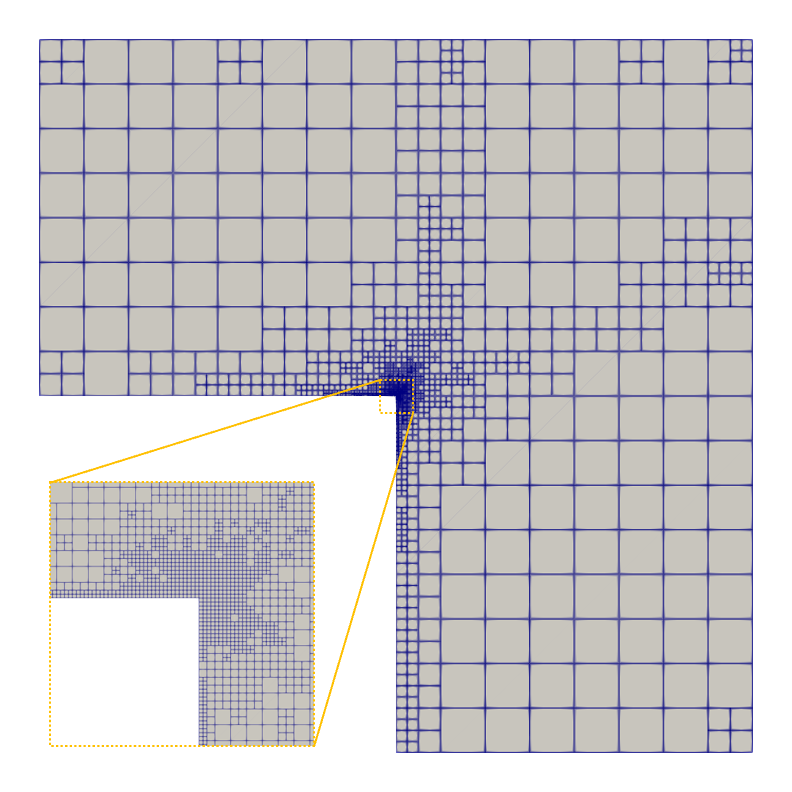}
        \caption{$\mathcal{T}_7^n$}
        \label{fig:lshape_network_indicator_final_mesh}
    \end{subfigure}
    \begin{subfigure}[t]{0.32\linewidth}
        \centering
        \includegraphics[height=0.88\textwidth]{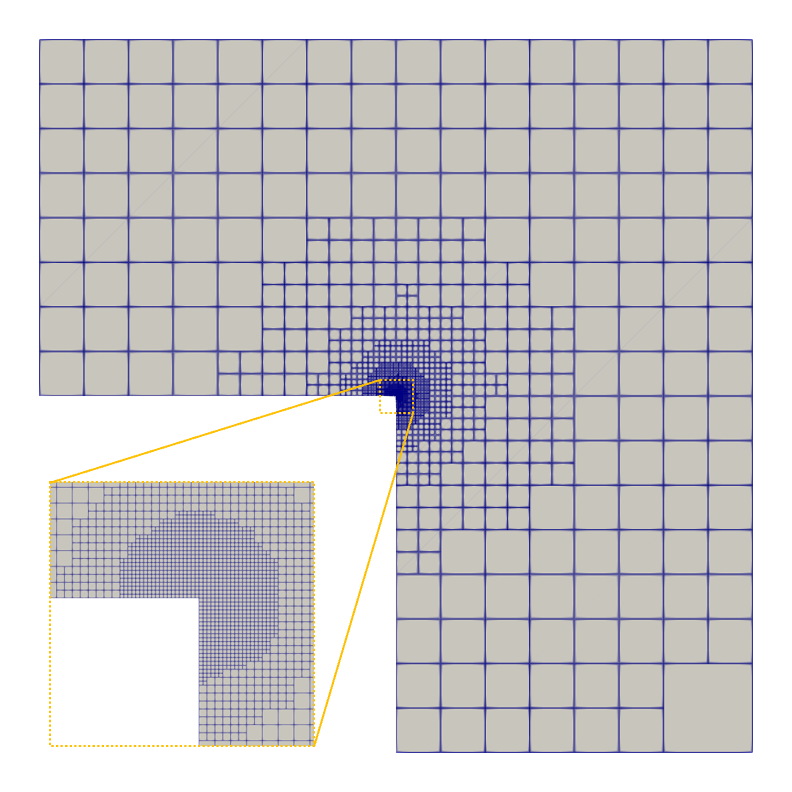}
        \caption{$\mathcal{T}_7^r$}
        \label{fig:lshape_real_indicator_final_mesh}
    \end{subfigure}
     
    \caption{Comparison of final meshes obtained by training of the $h$-adaptive \acp{feinn} for the 2D Fichera problem with singularity. The meshes result from 7 mesh adaptation steps using Kelly, network, and real error indicators, respectively.}
    \label{fig:lshape_feinn_final_meshes}
\end{figure}

\subsubsection{The effect of \ac{fe} space interpolation order} \label{subsubsec:lshaped_order_study}
So far, we have only conducted experiments using high interpolation order, i.e., $k_U=4$. We now extend our exploration to lower polynomial orders in solving the same problem, assessing the impact of $k_U$ on training results, while comparing the error decay of $h$-adaptive \acp{feinn} with \ac{fem} on uniformly refined meshes.
We experiment using the network error indicator and $k_U = 2$. To ensure a fair comparison, the \acp{dof} of the initial trial \ac{fe} spaces are kept constant across different values of $k_U$. This results in a finer mesh for $k_U=2$ compared to $k_U=4$. 
Besides, from our preliminary experiments, we find that a lower $k_U$ requires fewer training iterations. Therefore, for $k_U=2$, we determine the training iterations to be half of the iterations for $k_U=4$ (specified in \tab{tab:hyperparameters}).
It is important to point out that we exclude $k_U=1$ experiments, as the approximation of a smooth \ac{nn} with a linear \ac{fe} space is ineffective, thus making such experiments not very meaningful.

\fig{fig:lshape_network_indicator_order_study_error_convergence} compares the errors of the \acp{nn} and their \acp{feinn} versus number of \acp{dof} for various $k_U$. We use different colours to distinguish between orders in the figure.
It is evident that $h$-adaptive \acp{feinn} can more effectively reduce errors compared to uniform mesh refinement \ac{fem}, as all the solid lines eventually fall below the dashed lines representing the convergences of uniform \ac{fem}. 
Besides, we observe that the choice of $k_U$ has a limited impact on the convergence order. Across different $k_U$, the \ac{fem} lines exhibit almost identical slope, with only small distances between them. Similarly, the lines representing different $k_U$ in both \ac{feinn} and \ac{nn} results are also very close to each other.
In addition, for both interpolation orders, the $L^2$ and $H^1$ convergences of both \ac{feinn} and \ac{nn} are satisfactory. 
It is worth noting that the results in \fig{fig:lshape_network_indicator_order_study_error_convergence} are reproducible with regards to \ac{nn} initialisations and error indicators. Numerous independent runs with different \ac{nn} initial parameters and error indicators have validated the consistency of these results. We skip other similar results in the text for conciseness.

\begin{figure}[h]
    \centering
    \includegraphics[width=\textwidth]{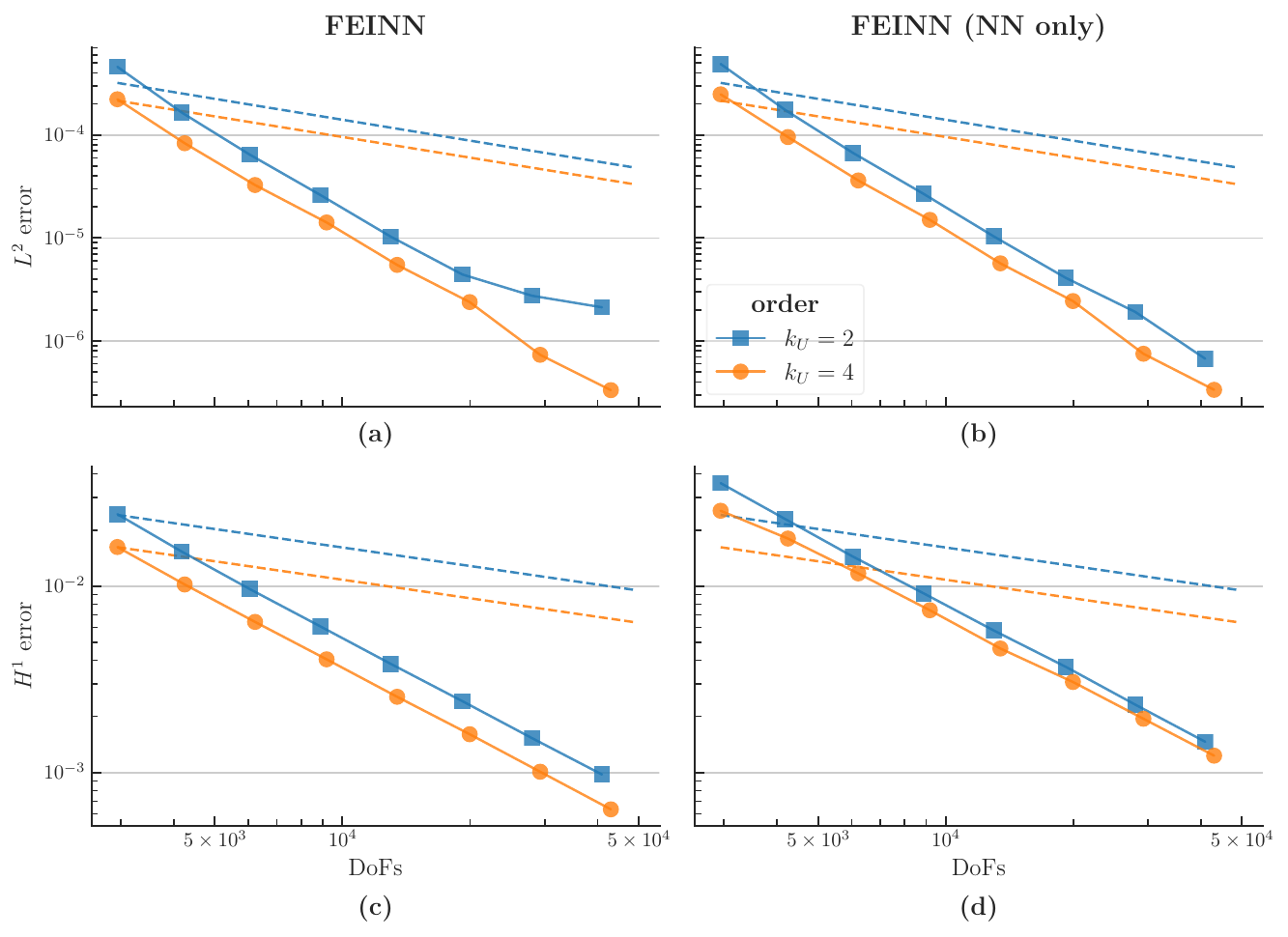}
    \caption{Convergence of $u^{id}$ in $L^2$ and $H^1$ errors for the 2D Fichera problem with singularity, using different interpolation orders. Network error indicator is used for mesh adaptation. The dashed lines depict the errors of the \ac{fem} solution with uniform refinement, obtained by using different polynomial orders. See caption of \fig{fig:arc_indicator_study_error_convergence} for other details in this figure.}
    \label{fig:lshape_network_indicator_order_study_error_convergence}
\end{figure}

\subsection{The 3D problem with singularities} \label{subsec:singularities3d}

\begin{figure}[h]
    \begin{subfigure}[t]{0.32\linewidth}
        \centering
        \includegraphics[height=0.88\textwidth]{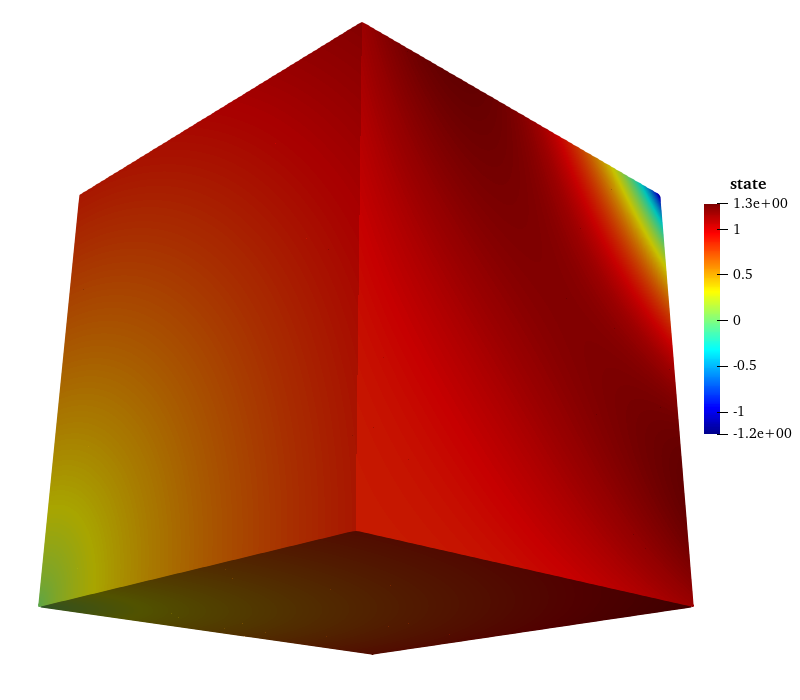}
        \caption{$u$}
        \label{fig:singular3d_true_state}
    \end{subfigure}
    \begin{subfigure}[t]{0.32\linewidth}
        \centering
        \includegraphics[height=0.88\textwidth]{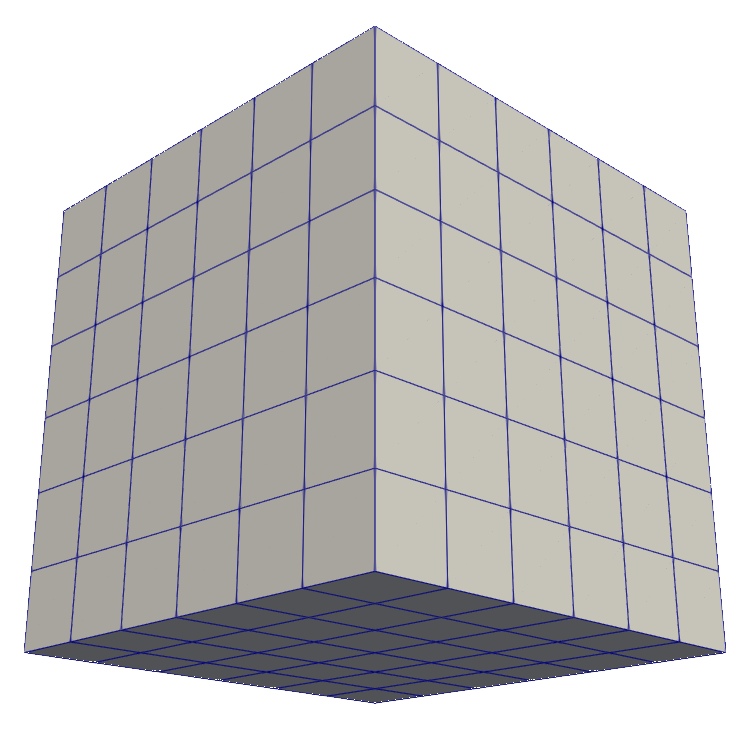}
        \caption{Initial mesh}
        \label{fig:singular3d_order4_init_mesh}
    \end{subfigure}
    \begin{subfigure}[t]{0.32\linewidth}
        \centering
        \includegraphics[height=0.88\textwidth]{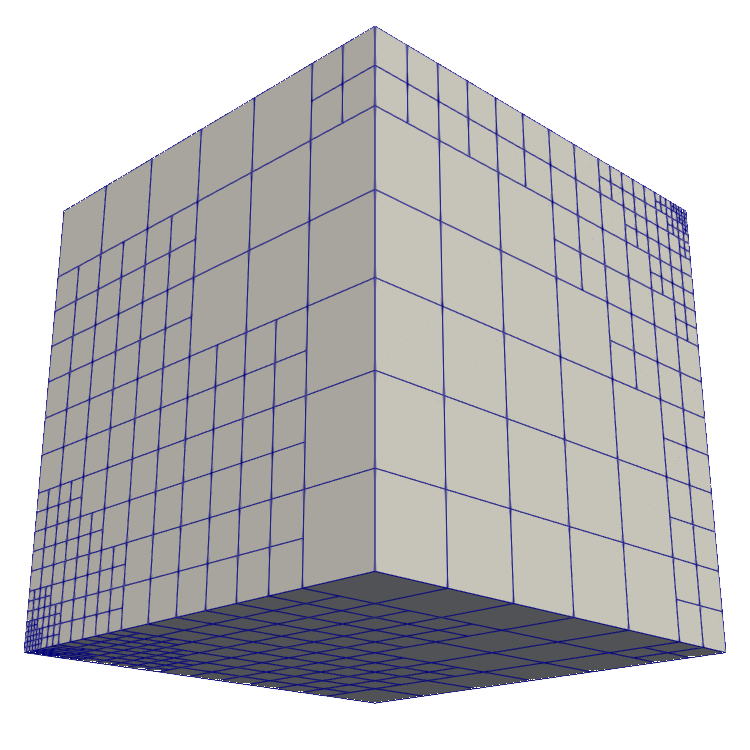}
        \caption{$h$-Adaptive \ac{fem} final mesh}
        \label{fig:singular3d_order4_fem_final_mesh}
    \end{subfigure}
     
    \caption{The true solution, initial mesh, and $h$-adaptive \ac{fem} final mesh for the 3D problem featuring singularities. The final mesh is obtained through the $h$-adaptive \ac{fem} with 4 mesh adaptation steps using the real \ac{fem} error as the error indicator.}
    \label{fig:singular3d_state_and_meshes}
\end{figure}

In our final experiment, we showcase the capability of our method in addressing 3D problems by considering a \ac{pde} defined on the unit cube $\Omega = [0,1]^3$. We consider an initial mesh $\mathcal{C}$ of $3^3$ cells. 
We pick $f$ and $g$ such that the true solution (\fig{fig:singular3d_state_and_meshes}a) is:
\begin{equation*}
    u(x, y, z) = \sqrt[3]{x^2+y^2+z^2} \sin(2\arcsin(xyz) + \pi / 3).
\end{equation*}
Note that the derivative of $u$ has two singularities: one at the origin, stemming from the $\sqrt[3]{x^2+y^2+z^2}$ term, and the other at the point $[1,1,1]$, coming from the $\arcsin(xyz)$ term. Thus, the final mesh is expected to be denser around these two singular points, as shown in \fig{fig:singular3d_state_and_meshes}c.
We use $k_U=4$ and start with the initial mesh shown in \fig{fig:singular3d_state_and_meshes}b, performing 4 mesh adaptation steps. We note that the initial mesh is obtained from the application of 1 level of uniform mesh refinement to $\mathcal{C}$. For each error indicator, we conduct 10 independent runs with different initial \ac{nn} parameters.

\begin{figure}[h]
    \centering
    \includegraphics[width=\textwidth]{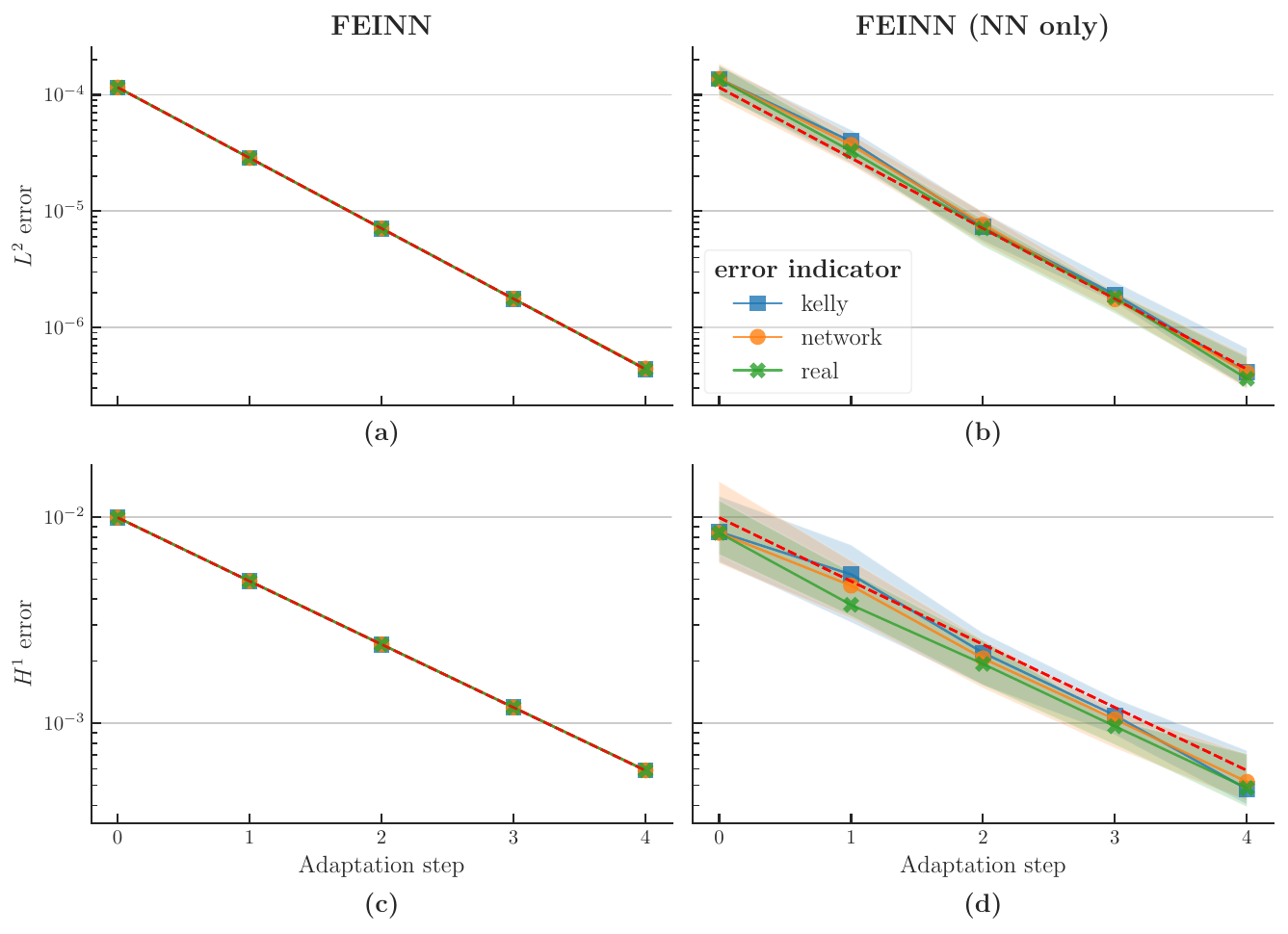}
    \caption{Convergence of $u^{id}$ in $L^2$ and $H^1$ errors for the 3D problem featuring singularities, using different error indicators. For each error indicator, 10 independent runs are performed. See caption of \fig{fig:arc_indicator_study_error_convergence} for details on the information being displayed in this figure.}
    \label{fig:singular3d_indicator_study_error_convergence}
\end{figure}

The error convergence of both \acp{feinn} and \acp{nn}, as presented in \fig{fig:singular3d_indicator_study_error_convergence}, illustrate an important finding: all three error indicators are equally effective in solving this 3D problem.
Notably, in the case of \acp{feinn}, both $L^2$ errors (\fig{fig:singular3d_indicator_study_error_convergence}a) and $H^1$ errors (\fig{fig:singular3d_indicator_study_error_convergence}c) are so closely aligned with the red dashed \ac{fem} lines that they appear visually indistinguishable. This highlights the exceptional accuracy of \acp{feinn} in mirroring the $h$-adaptive \ac{fem} solution, with negligible impact from \ac{nn} initialisations. 
Furthermore, the generalisation capability of \acp{nn} is convincing, with both $L^2$ errors (\fig{fig:singular3d_indicator_study_error_convergence}b) and $H^1$ errors (\fig{fig:singular3d_indicator_study_error_convergence}d) closely tracking the \ac{fem} lines. While there is observable variance in these errors, the associated bands fall within an acceptable range and are not excessively wide.
In short, the results confirm the ability of $h$-adaptive \acp{feinn} in effectively attacking \acp{pde} with singularities in 3D.

\begin{figure}[h]
    \begin{subfigure}[t]{0.32\linewidth}
        \centering
        \includegraphics[height=0.88\textwidth]{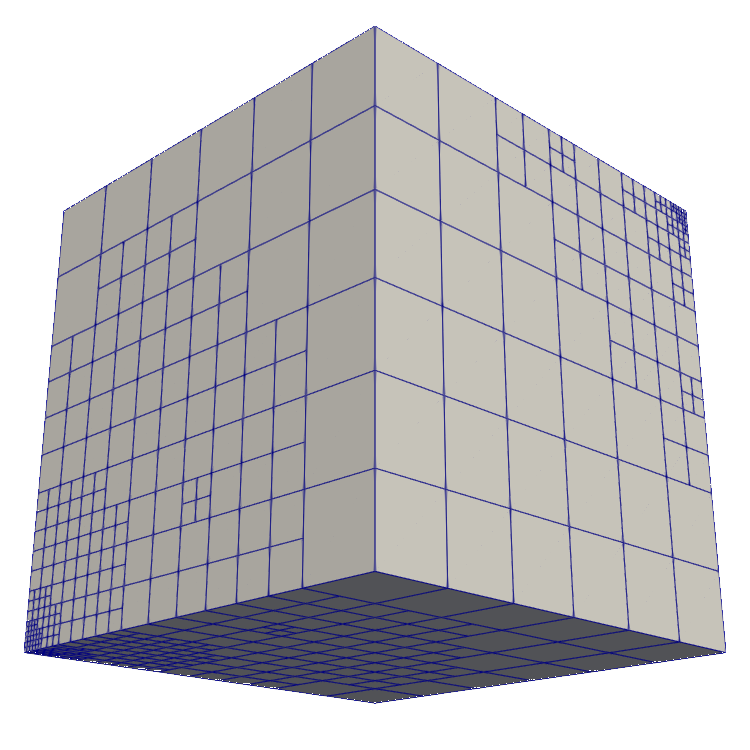}
        \caption{$\mathcal{T}_4^k$}
        \label{fig:singular3d_kelly_indicator_final_mesh}
    \end{subfigure}
    \begin{subfigure}[t]{0.32\linewidth}
        \centering
        \includegraphics[height=0.88\textwidth]{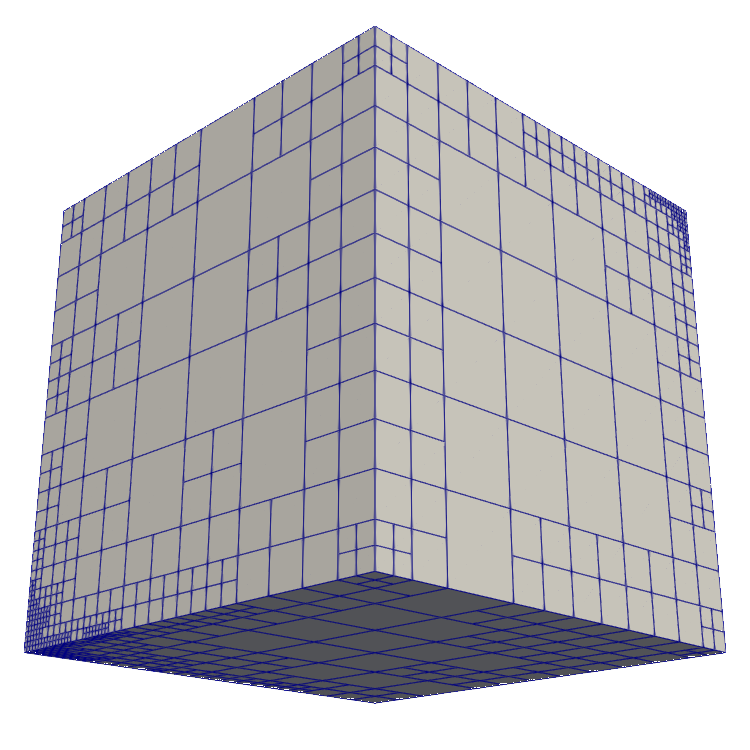}
        \caption{$\mathcal{T}_4^n$}
        \label{fig:singular3d_network_indicator_final_mesh}
    \end{subfigure}
    \begin{subfigure}[t]{0.32\linewidth}
        \centering
        \includegraphics[height=0.88\textwidth]{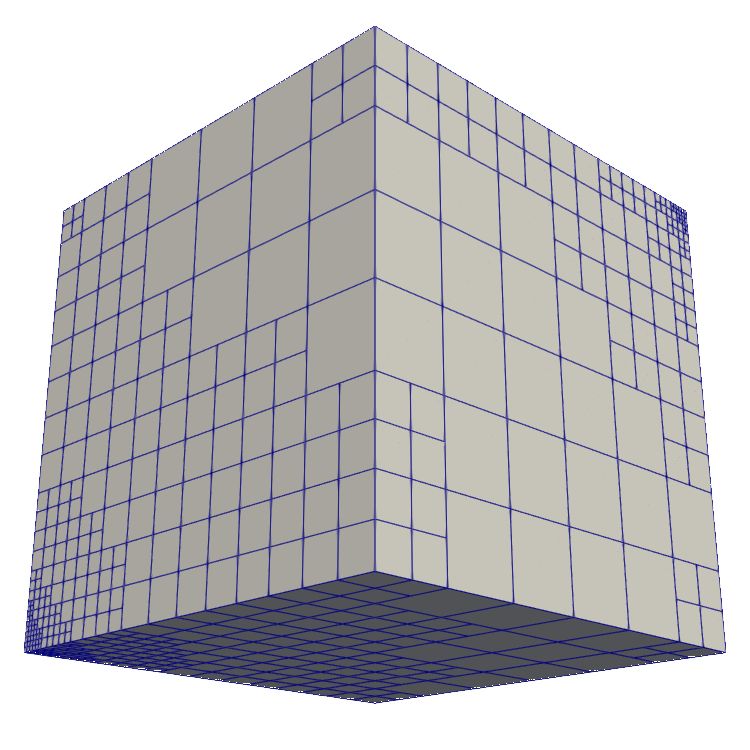}
        \caption{$\mathcal{T}_4^r$}
        \label{fig:singular3d_real_indicator_final_mesh}
    \end{subfigure}
     
    \caption{Comparison of final meshes obtained by training of the $h$-adaptive \acp{feinn} for the 3D problem with singularities. The meshes result from 4 mesh adaptation steps using Kelly, network, and real error indicators, respectively.}
    \label{fig:singular3d_feinn_final_meshes}
\end{figure}

\fig{fig:singular3d_feinn_final_meshes} displays the final meshes obtained using different error indicators. Notably, all three error indicators effectively identify the singular points and refine the elements around them, despite slight variations in the mesh patterns.
In particular, $\mathcal{T}_4^k$ (\fig{fig:singular3d_feinn_final_meshes}a) with 1,819 elements and 102,587 \acp{dof} and $\mathcal{T}_4^r$ (\fig{fig:singular3d_feinn_final_meshes}c) with 1,819 elements and 101,835 \acp{dof} look very similar to the $h$-adaptive \ac{fem} final mesh (\fig{fig:singular3d_state_and_meshes}c) with 1,826 elements and 102,811 \acp{dof}, while $\mathcal{T}_4^n$ (\fig{fig:singular3d_feinn_final_meshes}b) with 2,050 elements and 109,771 \acp{dof} tends to refine the elements along the edges. 
Relatively, the number of elements and \acp{dof} for these three meshes are very close, indicating similar computational costs for training with different error indicators.
With a high interpolation order ($k_U=4$) and a fine initial mesh, the specific patterns of the resulting meshes are less critical, as long as the elements are refined around singular points. 
It is noteworthy that, even in the existence of two singular points, the network error indicator demonstrates comparable performance to the well-known Kelly error indicator.
\section{Conclusions} \label{sec:conclusions}
In this work, we propose the $h$-adaptive \ac{feinn} method, an extension of the \ac{feinn} method~\cite{Badia2024} that incorporates adaptive meshes. The method inherits the desired properties of \acp{feinn}, namely the applicability to general geometries, simple imposition of Dirichlet boundary conditions and accurate integration of the interpolated weak residual. Furthermore, since the interpolation of the \ac{nn} is performed on an adaptive \ac{fe} space, the method preserves the non-linear capabilities of the \ac{nn} to capture sharp gradients and singularities. The dual norm of the finite element residual in the loss function enhances training robustness and accelerates convergence. We propose a train, estimate, mark and adapt strategy to train the \ac{feinn} and adapt the mesh simultaneously. We have carried out a detailed numerical analysis of the method, proving a priori error estimates depending on the expressiveness of the neural network compared to the interpolation mesh.

Our numerical findings confirm the capability of $h$-adaptive \acp{feinn} in capturing sharp gradients and singularities in 2D and 3D cases. 
Remarkably, in the 2D arc wavefront experiments, \ac{nn} exhibit excellent generalisation, surpassing \ac{fem} solutions by two orders of magnitude in terms of $H^1$ error at the final mesh adaptation step.
Additionally, our results verify the substantial improvements in training robustness and convergence with the introduction of preconditioning techniques.
The 2D L-shaped singularity experiments showcase the superior error reduction capabilities of $h$-adaptive \acp{feinn} compared to uniform refinement \ac{fem} for various interpolation orders. 
Furthermore, we demonstrate the method's capability to capture multiple singularities in the 3D numerical experiments. Alongside the development, we introduce a \ac{nn} error indicator and conduct a thorough comparison with the standard Kelly error indicator in our experiments. The results show that the \ac{nn} error indicator is competitive with the Kelly error indicator in terms of mesh adaptation efficiency and accuracy.

There are several promising directions for future exploration. One research avenue is to examine the effectiveness of the adaptive training strategy in solving inverse problems. In addition, we aim to investigate the applicability of the proposed approach in attacking problems featuring singularities defined in $H(\mathrm{curl})$ and $H(\mathrm{div})$ spaces~\cite{OLM2019}. Towards increasing applicability, we would also like to explore the combination of $h$-adaptive \acp{feinn} with unfitted \ac{fe} spaces~\cite{Badia2021}, which can be used to tackle problems with complex interfaces and geometries. Future work will also focus on extending the framework to $p$ and $r$-refinement strategies, as well as to transient \acp{pde}. 
Finally, a mathematical analysis of generalisation error bounds for the \ac{feinn}-trained \acp{nn} is a promising direction for future research.

\section{Acknowledgments}
This research was partially funded by the Australian Government through the Australian Research Council
(project numbers DP210103092 and DP220103160). This work was also supported by computational resources provided by the Australian Government through NCI under the NCMAS and ANU Merit Allocation Schemes. W. Li acknowledges the Monash Graduate Scholarship from Monash University, and the support from the Laboratory for Turbulence Research in Aerospace and Combustion (LTRAC) at Monash University through the use of their HPC Clusters.

\printbibliography

\end{document}